\documentclass[a4paper]{article}
\usepackage{graphicx}
\usepackage[utf8]{inputenc}
\usepackage[fleqn]{mathtools}
\usepackage{amsmath}
\usepackage{amssymb}
\usepackage{amsfonts}
\usepackage{fleqn}
\usepackage{psfrag}
\usepackage{mathrsfs}
\usepackage{amssymb}
\usepackage{amsthm}
\usepackage{algorithm}
\usepackage{algpseudocode}
\usepackage[parfill]{parskip}
 \usepackage{xr-hyper}
   \usepackage{hyperref}
\usepackage{pifont}
\usepackage[toc,page]{appendix} 
\makeatletter
\def\BState{\State\hskip-\ALG@thistlm}
\makeatother

\usepackage{float}
\newfloat{algorithm}{H}{lop}

\usepackage{pifont}
 \newtheorem{theorem}{Theorem}[section]
 \newtheorem{lemma}[theorem]{Lemma}

 \theoremstyle{definition}
\newtheorem{definition}{Definition}[section]
\newcommand{\vertiii}[1]{{\left\vert\kern-0.25ex\left\vert\kern-0.25ex\left\vert #1 
    \right\vert\kern-0.25ex\right\vert\kern-0.25ex\right\vert}}
\begin{document}
\title{Interior-exterior penalty approach for solving elasto-hydrodynamic lubrication problem: Part I}
\date{}
\author{Peeyush Singh$^{*}$}
\maketitle
\thispagestyle{empty}
\begin{abstract}
A new interior-exterior penalty method for solving quasi-variational inequality
and pseudo-monotone operator arising in two-dimensional point contact problem
is analyzed and developed in discontinuous Galerkin finite volume framework. 
In this article, we show that optimal error estimate in $H^{1}$ and $L^{2}$ norm is achieved under a light
load parameter condition. In addition, article provide a
complete algorithm to tackle all numerical complexities appear in the solution procedure.
We obtain results for moderate loaded conditions which is discussed at the end of the section.
This method is well suited for solving elasto-hydrodynamic lubrication 
line as well as point contact problems and can probably be treated as commercial software.
Furthermore, results give a hope for the further development of the scheme for highly loaded condition
appeared in a more realistic operating situation which will be discussed in part II.
\end{abstract}
\vspace{0.2in}
\\
\noindent{$^{*}$ \small Tata Institute of Fundamental Research CAM Banglore-208016, India}\\
Mobile no: +919793585195 \\
{e-mail: peeyush@tifrbng.res.in, peeyushs8@gmail.com}
\newpage
\maketitle
\section{Introduction}
The motivation behind the present study is to better understand theoretical and numerical aspects of partial differential equation (PDE) of elasto-hydrodynamic lubrication
(EHL) problems using discontinuous Galerkin finite volume method (DG-FVM) setting.
In particular, these numerical methods can derive from a firm theoretical foundation and understanding similar to finite element \cite{oden1985} and finite difference
see for example \cite{chouye},\cite{chou} \cite{chowkwak}, \cite{vassilevski}.
Finite volume method (FVM) formulation obtained by integrating the PDE over a control volume.
Due to its natural conservation property, flexibility and parallelizability FVM is commonly accepted in 
many realistic practical problems such as fluid mechanics computations and hyperbolic conservation laws which have minimum regularity of solution in nature.
It is also quite natural to assume the advantage of nonconforming or DG finite element method (see for example 
\cite{lion},\cite{arnold}, \cite{aubin}, \cite{suli} \cite{babuska},\cite{wheeler}, \cite{douglas},\cite{nitsche}, \cite{riviere} ) can be applied into DG-FVM
(see for example \cite{ye},\cite{chouye}). 
However, there are hardly any numerical results on DG-FVM for solving nonlinear variational inequalities or for solving EHL model problem. Therefore
in this article, an attempt has been made to establish theoretical framework such as convergence and error estimate for DG-FVM for solving EHL model problem 
with the help of interior-exterior penalty procedure. So far it was very ambiguous to prove the connection of exterior penalty in DG-FVM setting 
to capture free boundary. One key point analysis is needed to make a natural connection which later helps to prove convergence and error estimate for not only 
EHL problem but also general variational inequality. However, in this discussion, we will center around only for EHL study more practical result discussion will be given in the second part of this paper. 
\subsection{Model Problem}
Consider strongly nonlinear EHL model problem of a ball rolling in the positive $x$-direction gives rise to a variational inequality defined below as
\begin{figure}
\centering
\includegraphics[width=2.5in, height=2.5in, angle=0]{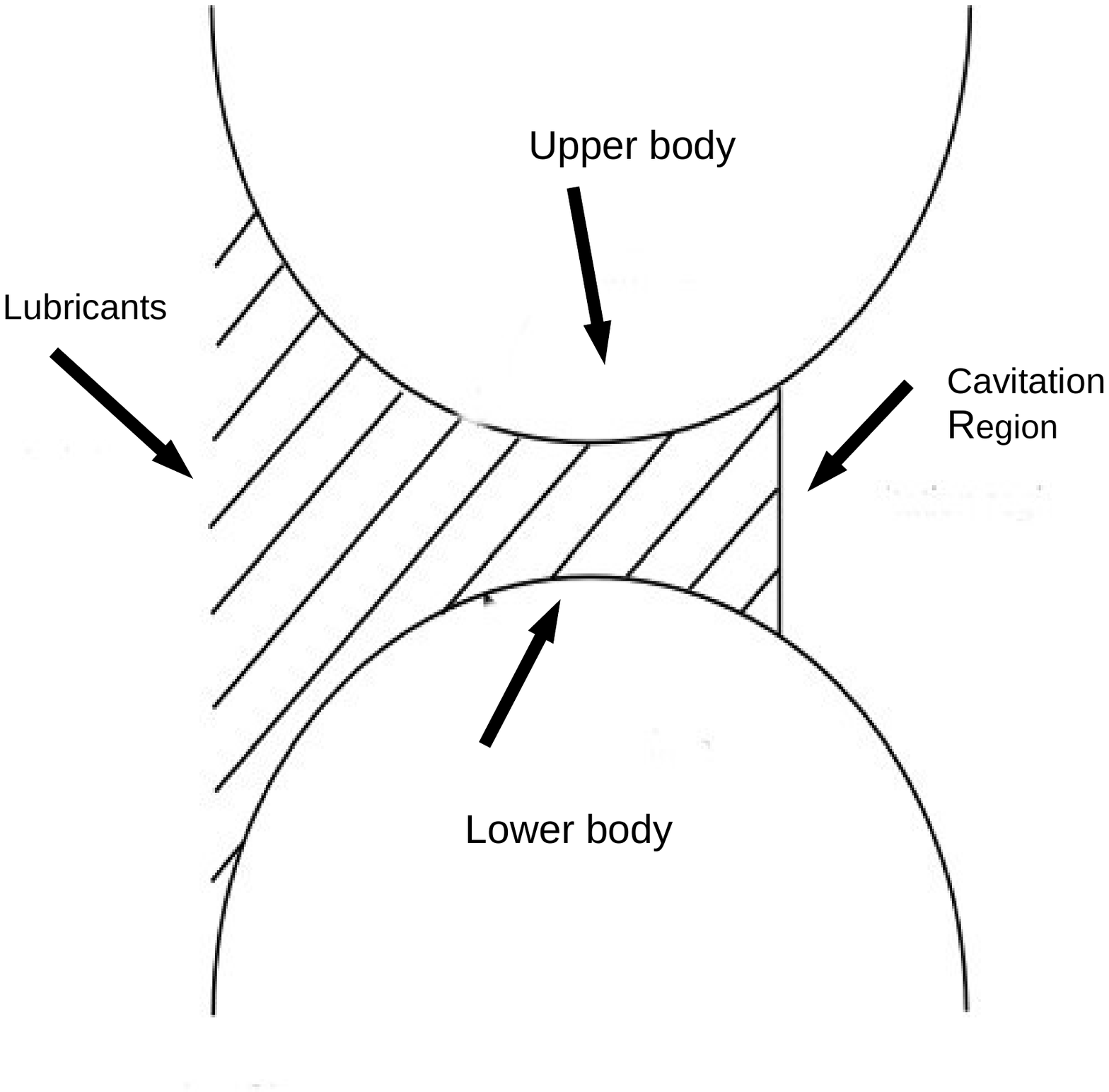}
\caption{Undeformed surface body}
\label{fig:undefm}
\end{figure}
\begin{figure}[h]
\centering
\includegraphics[width=2.5in, height=2.5in, angle=0]{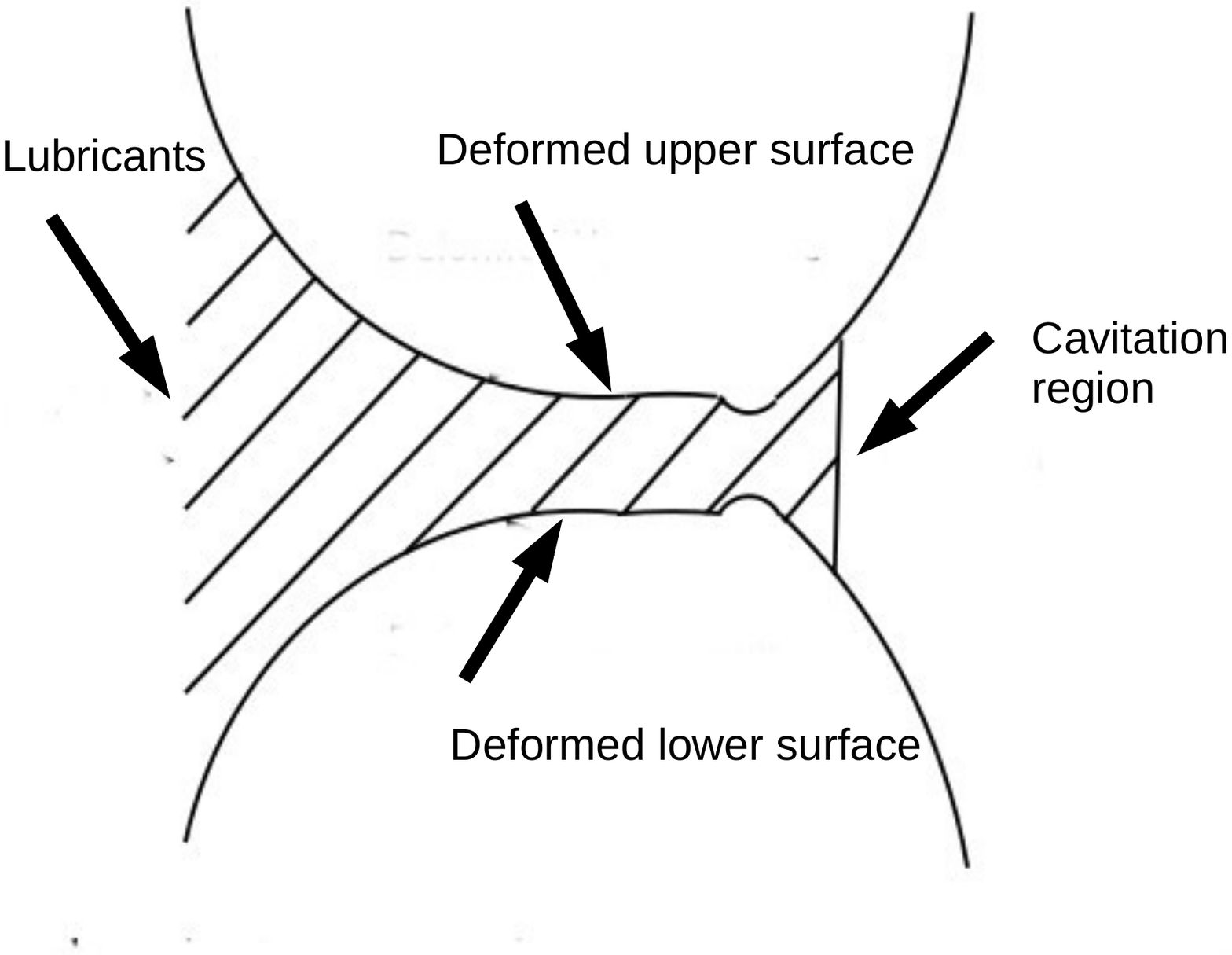}
\caption{Deformed surface body}
\label{fig:defm}
\end{figure}
\begin{align}\label{eq:1}
 \frac{\partial }{\partial x} \Big(\epsilon^{*} \frac {\partial u}{\partial x}\Big)+
 \frac{\partial }{\partial y} \Big(\epsilon^{*} \frac {\partial u}{\partial y}\Big)
 \le \frac {\partial (\rho h)}{\partial x}
\end{align}
\begin{align}\label{eq:2}
 u\ge 0
\end{align}
\begin{align}\label{eq:3}
 u.\Big[\frac{\partial }{\partial x} \Big(\epsilon^{*} \frac {\partial u}{\partial x}\Big)+
 \frac{\partial }{\partial y} \Big(\epsilon^{*} \frac {\partial u}{\partial y}\Big)
 -\frac {\partial (\rho h)}{\partial x}\Big] = 0,
\end{align}
where $u$ is pressure of liquid and $\rho,\epsilon^{*}=\frac{\rho h_{d}^{3}}{\eta}$ are defined in appendix \ref{appendix:pvalue}.
We consider above nonlinear variational inequality in a bounded, but large domain $\Omega$.
Since $u$ is small on $\partial \Omega$, it seems natural to impose the boundary condition
 \begin{align}\label{eq:4}
  u= 0 \quad \text{on} \quad \partial \Omega
 \end{align}
The film thickness equation is in dimensionless form is written as follows
\small
\begin{equation}\label{eq:5}
h_{d}(x,y) = h_{00}+\frac{x^{2}}{2}+\frac{y^{2}}{2} 
+\frac{2}{\pi^{2}}\int_{\Omega}\frac{u(x^{'},y^{'})dx^{'}dy^{'}}{\sqrt{(x-x^{'})^2+(y-y^{'})^2}}
\end{equation}
\normalsize
where $h_{00}$ is an integration constant.\\
The dimensionless force balance equation is defined as follows
\begin{align}\label{eq:6}
 \int_{-\infty}^{\infty} \int_{-\infty}^{\infty}u(x',y') dx'dy' = \frac{3\pi}{2}
\end{align}
Consider the ball is elastic whenever load is large enough. Then system \ref{eq:1}--\ref{eq:6} forms an Elasto-hydrodynamic Lubrication.
Schematic diagrams of EHL model is given in \ref{fig:undefm} and \ref{fig:defm} in the form of undeformed and deformed contacting body structure respectively.\\
The remainder of the article is organized as follows. In section \ref{section:vi} variational inequality and its notation is established;
Furthermore, existence results are proved for our model problem; In section.~\ref{section:dgfvm} DG-FVM notation 
and the proposed method is demonstrated; In section.~\ref{section:error} Error estimates are proved in $L^2$ and $H^{1}$ norm;
In section.~\ref{section:ntest} numerical experiment and graphical results are provided;
At last section.~\ref{section:con} conclusion and future direction is mentioned.
\section{Variational Inequality}\label{section:vi}
We consider space $\mathscr{V} = H^{1}_{0}(\Omega)$ and its dual space as $\mathscr{V}^{*}= (H^{1}_{0}(\Omega))^{*} = H^{-1}(\Omega)$.
Also define notion $\langle.,.\rangle$ as duality pairing on $\mathscr{V}^{*} \times \mathscr{V} $.
Further assume that $\mathscr{C}$ is closed convex subset of $\mathscr{V}$ defined by 
\begin{align}\label{eq:11}
\mathscr{C} = \Big\{  v \in \mathscr{V}: v \ge 0  \text{ a.e. } \in \Omega \Big\}
\end{align}
Additionally, we define the operator $\mathscr{T}$ as 
\begin{align}\label{eq:12}
 \mathscr{T}:u\rightarrow -\Big[\frac{\partial }{\partial x} \Big(\epsilon^{*} \frac {\partial u}{\partial x}\Big)+
 \frac{\partial }{\partial y} \Big(\epsilon^{*} \frac {\partial u}{\partial y}\Big)\Big]
+\frac {\partial (\rho h_{d})}{\partial x}
\end{align}
Then, for a given $f \in \mathscr{V}^{*}$, the problem of finding an element $u \in \mathscr{C}$ such that
\begin{align}\label{eq:14}
 \langle \mathscr{T}(u)-f,v-u \rangle \ge 0, \quad \forall v \in \mathscr{C}.
\end{align}
Throughout in the article we shall assume that there exists $\delta >0$ and $K_{*} >0$ such that
\begin{align}
\frac{\partial \epsilon(\varsigma)}{\partial u}.\nabla u \ge K_{*}|u|^{2}\quad \forall \varsigma \in \Omega \quad
\forall \varsigma \in \mathcal{Z}_{\delta}
\end{align}
\begin{definition}
 Operator $\mathscr{T}: \mathscr{C} \subset \mathscr{V}\rightarrow \mathscr{V}^{*}$ is said to be pseudo-monotone if 
 $\mathscr{T}$ is a bounded operator and whenever $u_{k}\rightharpoonup u$ in $\mathscr{V}$ as $k \rightarrow \infty$
 and 
 \begin{align}\label{eq:b14}
 \lim_{k \rightarrow \infty}\sup\langle \mathscr{T}(u_{k}),u_{k}-u\rangle \le 0.
\end{align}
it follows that for all $v \in \mathscr{C}$
 \begin{align}\label{eq:15}
 \lim_{k \rightarrow \infty}\inf\langle \mathscr{T}(u_{k}),u-v\rangle 
 \ge \langle\mathscr{T}(u),u-v\rangle.
\end{align}
\end{definition}
\begin{definition}
Operator $ \mathscr{T}: \mathscr{V} \rightarrow \mathscr{V}^{*}$ is said to be hemi-continuous 
if and only if the function $\phi: t \longmapsto \langle \mathscr{T}(tx+(1-t)y),x-y\rangle $ is
continuous on $[0,1] \quad \forall x,y \in \mathscr{V}$.
\end{definition}
On this context the following existence theorem has been proved by Oden and Wu \cite{oden1985} by assuming constant density 
and constant viscosity of the lubricant.
However, idea is easily extend-able for more realistic operating condition in which density 
and viscosity of the lubricant are depend on its applied pressure see Appendix.~\ref{appendix:pvalue}.
A straight forward modification of the analysis of \cite{oden1985} yields the theorem below and so we will omit the proof.
\begin{theorem}\cite{oden1985}
Let $\mathscr{C}(\neq \emptyset)$ be a closed, convex subset of a reflexive Banach space $\mathscr{V}$ and let $\mathscr{T}: \mathscr{C} \subset \mathscr{V} \rightarrow \mathscr{V}^{*}$ be a pseudo-monotone,
bounded, and coercive operator from $\mathscr{C}$ into the dual $\mathscr{V}^{*}$ of $\mathscr{V}$, in the sense that there exists $y \in \mathscr{C}$ such that
 \begin{align}\label{eq:18}
 \text{lim}_{||x||\rightarrow \infty}\frac{\langle  \mathscr{T}(x),x-y\rangle}{||x||} = \infty.
 \end{align}
 Let $f$ be given in $\mathscr{V}^{*}$ then there exists at least one $u \in \mathscr{C}$ such that 
 \begin{align}\label{eq:19}
  \langle \mathscr{T}(x)-f,y-x\rangle \ge 0 \quad \forall y \in \mathscr{C}.
 \end{align}
\end{theorem}
In the next section, we will give a complete formulation as well as will give theoretical justification for existence of 
our model problem in discrete computed setting.
\section{Discrete Formulation of DG-FVM}\label{section:dgfvm}
\begin{figure}
\centering
\includegraphics[width=2.0in, height=2.0in, angle=0]{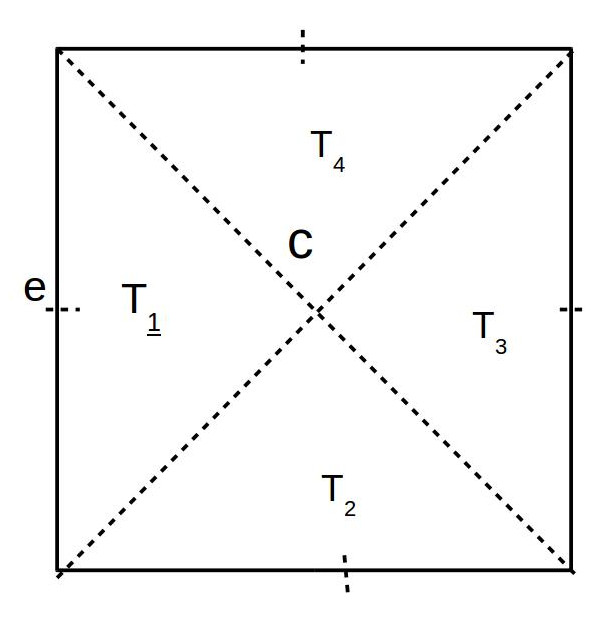}
\caption{Rectangular partition}
\label{fig:elm}
\end{figure}
We define finite dimensional space associated with $\mathscr{R}_h$  for trial functions as
\begin{equation}\label{eq:26}
\mathscr{V}_h = \{ v \in L^{2}(\Omega): v|_K \in \mathcal{S}_{1}(K), v|_{\partial \Omega}=0 \quad \forall K \in \mathscr{R}_h \}.
\end{equation}
Define the finite dimensional space $\mathscr{W}_h$ for test functions associated with the dual partition $\mathscr{M}_h$ as
\begin{equation}\label{eq:27}
\mathscr{W}_h = \{ q \in L^{2}(\Omega): q|_T \in \mathcal{S}_{0}(T),q|_{\partial \Omega}=0 \quad \forall T \in \mathscr{M}_h \},
\end{equation}
where $\mathcal{S}_{l}(T)$ consist of all the polynomials with degree less than or equal to $l$ defined on $T$.\\
Let $\mathscr{V}(h) = \mathscr{V}_h + H^2(\Omega) \cap H^{1}_{0}(\Omega)$. Define a mapping 
\begin{align}\label{eq:28}
\gamma : \mathscr{V}(h) \longmapsto \mathscr{W}_{h} \quad
\gamma v|_T = \frac{1}{h_e} \int_e v|_T ds , \quad T \in \mathscr{M}_h.
\end{align}
Let $T_{j} \in \mathscr{M}_h (j=1,2,3,4)$ be four triangles in $K \in \mathscr{R}_h$.
Let $e$ be an interior edge shared by two elements $K_1$ and $K_2$ in $\mathscr{R}_h$ and let $\bold{n_1} $ and $\bold{n_2}$ be unit
normal vectors on $e$ pointing exterior to $K_1$ and $K_2$ respectively.
We define average \{.\} and jump [.] on $e$ for scalar $q$ and vector $w$, respectively, as (\cite{arnold})
\[\{q\} = \frac{1}{2}(q|_{\partial T_1}+ q|_{\partial T_2}), \quad 
[q] = (q|_{\partial T_1}\bold{n_1}+ q|_{\partial T_2}\bold{n_2})\]
\[\{w\} = \frac{1}{2}(w|_{\partial T_1}+ w|_{\partial T_2}), \quad 
[w] = (w|_{\partial T_1}\bold{n_1}+ w|_{\partial T_2}\bold{n_2})\]
If $e$ is a edge on the boundary of $\Omega$, we define
${q} = q,\quad [w] = w.\bold{n}$.
Let $\Gamma$ denote the union of the boundaries of the triangle $K$ of $\mathscr{R}_h$ and $\Gamma_{0}:= \Gamma\diagdown\partial \Omega$.
\subsection{Weak Formulation}
Reconsider the problem of the type 
\begin{align}
\frac{\partial }{\partial x} \Big(\epsilon^{*} \frac {\partial u}{\partial x}\Big)+
 \frac{\partial }{\partial y} \Big(\epsilon^{*} \frac {\partial u}{\partial y}\Big)
 -\frac {\partial (\rho h)}{\partial x}=0 \quad \text{in } \Omega\\
 u =0 \quad \text{on } \partial \Omega,
\end{align}
where all notation has their usual meaning.\\
For given $u,v \in H^{2}(\Omega)$ and for fixed value of $\Phi \in H^{2}(\Omega)$, define bilinear form as
\begin{align}\label{eq:33}
 \langle \mathscr{T}(\Phi;u),v\rangle 
 = \sum\limits_{K \in \mathscr{R}} \sum\limits_{j=1}^{4} \int_{A_{j+1}CA_{j}} \epsilon(\Phi) \nabla u.\bold{n}\gamma v ds \nonumber \\
+ \sum\limits_{e \in \Gamma}\int_{e} [v]\{ \epsilon(\Phi) \nabla u.\bold{n}\}ds 
+ \alpha_{1} \sum\limits_{e \in \Gamma}[\gamma u]_e[v]_e \nonumber \\
-\sum\limits_{K \in \mathscr{R}_h} \sum\limits_{j=1}^{4} \int_{A_{j+1}CA_{j}}
(\rho(\Phi) h_{d}(u)).(\beta .\bold{n})\gamma v ds 
- \sum\limits_{e \in \Gamma} \int_{e} [v] \{ (\rho(\Phi) h_{d}(u)).(\beta.\bold{n})\} ds.
\end{align}
We define the following mesh dependent norm $\vertiii{.}$ and $\vertiii{.}_{\nu}$ as
\begin{align}\label{eq:36}
\vertiii{v}^{2}= |v|^{2}_{1,h}+\sum_{e}|\gamma v|_{e}^{2}\\
\vertiii{v}^{2}_{\nu}= |v|^{2}_{1,h}+\sum_{e}h_{e}\int_{e}\Big\{ \frac{\partial v }{\partial \nu}\Big\}^{2} ds 
+\sum_{e}|\gamma v|_{e}^{2},\text{ where }|v|^{2}_{1,h}=\sum_{K}|v|_{1,K}^{2}.
\end{align}
Now we will state few lemmas and inequalities without proof which will be later helpful in our subsequent analysis.
\begin{lemma}
For $u \in H^{s}(K_{i})$, there exist a positive constant $C_{A}$ and an interpolation value $u_{I} \in \mathscr{V}_{h}$, such that
\begin{align}
 |\vert u-u_{I} \vert|_{s,K} \le C_{A}h^{2-s}|u|_{2,K},\quad s=0,1.
\end{align}
\end{lemma}
{\bf Trace inequality.} We state without proof the following trace inequality. Let $\phi \in H^{2}(K)$ and for an edge $e$ of $K$,
\begin{align}
 |\vert \phi |\vert_{e}^{2} \le C(h_{e}^{-1}|\phi|_{K}^{2}+h_{e}|\phi|_{1,K}^{2}).
\end{align}

\begin{lemma}
Let for any $u,v \in \mathscr{V}_{h}$, then we have following relation
\begin{align}\label{eq:37}
\langle h^{3}_{d}\rho e^{-au}\nabla_{h} u,\nabla_{h} v  \rangle \le
\langle  \mathscr{T}_{1}(u;u_{h}),v \rangle+C_{1}h \vertiii{u}\vertiii{v}, 
\end{align}
where $$\langle  \mathscr{T}_{1}(u;u_{h}),v \rangle
=\sum\limits_{K \in \mathscr{R}} \sum\limits_{j=1}^{4} \int_{A_{j+1}CA_{j}} h^{3}_{d}\rho e^{-au} \nabla u.\bold{n}\gamma v ds$$
\end{lemma}
\begin{proof}
Proof of lemma follows using similar argument as mentioned in \cite{ye},lemma 2.1.
\end{proof}
Next lemma provides us a bound of film thickness term and later helpful in proving coercivity and error analysis.
\begin{lemma}
 For $h_d$ defined in equation \ref{eq:5}, $0 < \beta_{*} < 1, s= 2-\beta_{*}/(1-\beta_{*})>2$ there exist $C_{1} \text{ and } C_{2}>0$ such that
 \begin{align}\label{eq:38}
  \max_{x,y \in \Omega}|h_{d}(u)| \le C_{1}+C_{2}\lVert u \rVert_{L^{s}} \quad 0 < \beta_{*} <1, \quad \forall (x,y) \in \bar{\Omega}.
 \end{align}
\end{lemma}
\begin{lemma}
 The operator $ \mathscr{T}$ defined in equation \ref{eq:34} is bounded as a map from $\mathscr{V}$ into $\mathscr{V}^{*}$.
\end{lemma}
\begin{lemma}
 The operator $\mathscr{T}$, defined in equation (21) is hemi-continuous, that is $\forall u, v,w \in \mathscr{V}$,
 \[ \lim_{t \rightarrow 0^{+}}\langle \mathscr{T}(u+tv), w \rangle=\langle \mathscr{T}(u), w \rangle.\]
\end{lemma}
\begin{lemma}
 The operator defined on equation (21) is coercive i.e. there is a constant $C$ independent of $h$ such that for $\alpha_{1}$ large enough and $h$ is small enough 
 \begin{align}\label{eq:47}
 \langle  \mathscr{T}(u;u_{h}),u_{h} \rangle \ge C \vertiii{u_{h}}^{2} \quad \forall u_{h} \in \mathscr{V}_{h}
 \end{align}
\end{lemma}
 \subsection{Exterior penalty solution approximation}
In this section, we introduce an exterior penalty term to regularize the inequality constraint \ref{eq:1}--\ref{eq:6}.
We define a exterior penalty operator ${\xi}:H_{0}^{1}(\Omega) \rightarrow H^{-1}$ as 
\begin{align}\label{eq:48}
 {\xi}(u) = u^{-}/\epsilon \quad \text{with } \epsilon > 0,
\end{align}
where $u^{-} = u-\max(u,0)=\dfrac{u-|u|}{2}$.
Let us define exterior penalty problem,
$(\mathscr{U}_{\epsilon})$: for $\epsilon > 0,\quad \text{find}\quad u_{\epsilon} \in \mathscr{V}_{h}$ such that
\begin{align}\label{eq:49}
 \langle \mathscr{T}(u_{\epsilon}^{n}), v \rangle + \langle {\xi}(u_{\epsilon}), 
 v \rangle/\epsilon = \langle f, v \rangle  \quad \forall v \in \mathscr{V}_{h},
\end{align}
Then we will show that there exist solutions $u_{\epsilon}^{n}\in \mathscr{V}_{h}$ (For proof of this we will refer 
to see \ref{appendix:relax}). This approach can be used in our DG-FVM case and modified discrete weak formulation is written as
\small
\begin{equation}\label{eq:52}
\langle \mathscr{T}_{1}(u),\gamma v\rangle+\dfrac{1}{{\varepsilon}}
\sum\limits_{K \in \mathscr{R}_h}\sum\limits_{j=1}^{4}\int_{T_{j} \subset K}u_{-}
\gamma vds-\langle \mathscr{T}_{2}(u),\gamma v\rangle=0,\forall v \in \mathscr{V}_h,
\end{equation}
\normalsize
where ${\varepsilon}$ is an arbitrary small positive number (${\varepsilon} = 1.0 \times 10^{-6}$).
\begin{lemma}
Penalty operator $\xi: \mathscr{V}\longmapsto \mathscr{V^{*}}$ is monotone, coercive and bounded.
\end{lemma}
\begin{proof}
Now define domains $\Omega_{1}=\{ x \in \Omega_{1}: u_{1} > 0 \}$ and $\Omega_{2}=\{ x \in \Omega_{2}: u_{2} > 0 \}$ 
and their compliments as $\Omega_{1}^{c} \text{ and } \Omega_{2}^{c}$ respectively. Also consider
 \begin{align}\label{eq:53}
 u_{i}=
 \begin{cases}
    u_{i} \in \Omega_{i}^{c} \quad \forall i=1,2 \\[2ex] 0 \in \Omega_{i} \quad \forall i=1,2.
  \end{cases}
  \end{align}
For proving monotonicity we consider
\[
 \langle \xi(u_{1})-\xi(u_{2}),u_{1} - u_{2} \rangle 
 =\sum_{K \in \mathscr{R}_{h}} \int_{K}u_{1}^{-}(\gamma(u_{1}-u_{2}))-u_{2}^{-}(\gamma(u_{1}-u_{2})) dx\]
 \[=\sum_{K \in \mathscr{R}_{h}} \int_{K}u_{1}^{-}(\gamma(u_{1}-u_{2})-(u_{1}-u_{2})+(u_{1}-u_{2}))dx\]
 \[-\sum_{K \in \mathscr{R}_{h}}\int_{K}u_{2}^{-}(\gamma(u_{1}-u_{2})-(u_{1}-u_{2})+(u_{1}-u_{2})) dx\] 
\[ =\sum_{K \in \mathscr{R}_{h}}\int_{K \cap \Omega_{1}^{c}}u_{1}^{-}(u_{1}-u_{2})dx
-\int_{K \cap \Omega_{2}^{c}}u_{2}^{-}(u_{1}-u_{2}) dx\]
 \[+\sum_{K \in \mathscr{R}_{h}}\int_{K\cap \Omega_{1}^{c}\cap \Omega_{2}^{c} }u_{1}^{-}(u_{1}-u_{2})-u_{2}^{-}(u_{1}-u_{2}) dx\] 
 \[+\sum_{K \in \mathscr{R}_{h}}\int_{K\cap \Omega_{1}\cap \Omega_{2}}u_{1}^{-}(u_{1}-u_{2})-u_{2}^{-}(u_{1}-u_{2}) dx\]
\[ =\sum_{K \in \mathscr{R}_{h}}\int_{K\cap\Omega_{1}^{c}}u_{1}^{-}(u_{1}-u_{2})-u_{2}^{-}(u_{1}-u_{2})dx\]
 \[-\sum_{K \in \mathscr{R}_{h}}\int_{K\cap\Omega_{2}^{c}}u_{1}^{-}(u_{1}-u_{2})-u_{2}^{-}(u_{1}-u_{2}) dx\]
\[+\sum_{K \in \mathscr{R}_{h}}\int_{K\cap \Omega_{1}^{c}\cap \Omega_{2}^{c} } (u_{1}-u_{2})^2dx \ge 0\]
Hence, operator is monotone. Also, coercivity follows from the fact that
\begin{align}\label{eq:54}
\langle\xi(u),u\rangle =\langle u^{-},u \rangle
=\sum_{K \in \mathscr{R}_{h}}\int_{K (u \le 0)} u^{-}\gamma udx
=\sum_{K \in \mathscr{R}_{h}}\int_{K (u \le 0)}u^{-}(\gamma u-u+u)dx \nonumber\\
=\sum_{K \in \mathscr{R}_{h}}\int_{K (u \le 0)} (u^{-})^{2}dx = |\lVert (u^{-}) \rVert|^{2}  \ge 0.
\end{align}
Furthermore, since 
\begin{align}\label{eq:55}
|\langle \xi(u),v \rangle| = | u^{-}\gamma v |\le |\lVert u\rVert||\lVert v \rVert|.
\end{align}
This implies that $\xi$ is bounded.
\end{proof}
\subsection{Linearization}
Let us consider a fix value of $w_{u} \in H^{2}(\Omega)$ and also take $w,v \in H^{2}(\Omega)$.
Furthermore, consider bilinear form $\mathscr{B}(w_{u};w,v)$ solving EHL problem defined in 1.1-1.6 as
\begin{align}
\mathscr{B}(w_{u};w,v):=
\sum\limits_{K \in \mathscr{R}} \sum\limits_{j=1}^{4} \int_{A_{j+1}CA_{j}} \epsilon(w_{u}) \nabla w.\bold{n}\gamma v ds \nonumber \\
+ \sum\limits_{e \in \Gamma}\int_{e} [v]\{ \epsilon(w_{u}) \nabla w.\bold{n}\}ds 
+ \alpha_{1} \sum\limits_{e \in \Gamma}[\gamma v]_e[\gamma w]_e+\beta_{1}\langle w^{-},v \rangle \nonumber \\
-\sum\limits_{K \in \mathscr{R}_h} \sum\limits_{j=1}^{4} \int_{A_{j+1}CA_{j}}
(\rho(w_{u}) h_{d}(x)).(\beta .\bold{n})\gamma v ds 
- \sum\limits_{e \in \Gamma} \int_{e} [v] \{ (\rho(w_{u}) h_{d}(x)).(\beta.\bold{n})\} ds
\end{align}
Now define weak formulation for solving DGFVEM for solving problem 1.1-1.6 as
find $u\in H^{2}(\Omega,\mathscr{R}_{h})$ such that 
\begin{align}
\mathscr{B}(u;u,v)=0.
\end{align}
Also $u_{h} \in \mathscr{V}_{h} \subset H^{2}(\Omega,\mathscr{R}_{h})$ so we have 
\begin{align}
\mathscr{B}(u;u,v)=\mathscr{B}(u_{h};u_{h},v_{h}) \quad \forall v_{h} \in \mathscr{V}_{h}.
\end{align}
Since we are solving highly non-linear type of operator and so an appropriate linearizion 
is required for further analysis. Therefore, we use following Taylor series expansion to linearize the problem as
\begin{align}
 \epsilon(w)=\epsilon(u)+\tilde{\epsilon}_{u}(w)(w-u),
\end{align}
where $\tilde{\epsilon}_{u}(w)=\int_{0}^{1}\epsilon_{u}(w +\tau [w-u]) d\tau$
and
\begin{align}
 \epsilon(w)=\epsilon(u)+{\epsilon}_{u}(w)(w-u)+\tilde{\epsilon}_{uu}(w)(w-u)^{2},
\end{align}
where $\tilde{\epsilon}_{uu}(w)=\int_{0}^{1}(1-\tau)\epsilon_{uu}(w +\tau [w-u]) d\tau$.
It is easy to check that $\tilde{\epsilon}_{u} \in C^{1}_{b}(\bar{\Omega},\mathcal{R})$ and 
$\tilde{\epsilon}_{uu} \in C^{0}_{b}(\bar{\Omega},\mathcal{R})$.\\
Now consider the following bilinear form $\bar{\mathscr{B}}(: ,.)$ as
\begin{align}
\bar{\mathscr{B}}(w_{u},w,v)= \mathscr{B}(w_{u},w,v)
+\sum_{K \in \mathscr{R}_{h}}\sum_{j=1}^{4}\int_{A_{j+1}CA_{j}}(\epsilon_{u}(w_{u})\nabla w_{u})w.\gamma v ds \nonumber \\
+\sum_{e \in \Gamma}\int_{e}[\gamma v]\Big\{\epsilon_{u}(w_{u})\nabla w_{u}w\Big\}ds
+\sum_{K \in \mathscr{R}_{h}}\sum_{j=1}^{4}\int_{A_{j+1}CA_{j}}\rho_{u}h_{d} w \vec{\beta}.\bold{n} \gamma v ds \nonumber \\
+\sum_{e \in \Gamma}\int_{e}[\gamma v]\Big\{\rho_{u}h_{d} w\Big\}ds.
\end{align}
It is easy to check that $\bar{\mathscr{B}}$ is linear in $w$ and $v$ and for fixed value of $w_{u} \in H^{2}(\Omega)$.
Also as $\epsilon(w_{u}) \in C^{2}_{b}(\bar{\Omega}, \mathcal{R})$ and $u \in C^{2}(\bar{\Omega})$, there is a unique solution
$w_{u} \in H^{2}(\Omega)$ to the following elliptic problem:
\begin{align}
 -\nabla.(\epsilon(u)\nabla \varphi+\epsilon_{u}\varphi \nabla u)+\nabla(\vec{\beta}(\rho h_{d}+\rho_{u} h_{d}\varphi))
 =\psi_{h} \text{ in } \Omega \nonumber \\
 \varphi=0 \text{ on } \partial \Omega.
\end{align}
and from well known elliptic regularity property we have 
\begin{align}
 \lVert \varphi \rVert_{H^{2}(\Omega)} \le C\lVert \psi_{h} \rVert 
\end{align}
Now for showing existence, uniqueness and for analyzing intermediate stage error analysis of discrete DGFVM solution we 
linearize weak formulation (35) around $\Pi_{h}u$. Let $e=u-u_{h}$ be an error term for exact and approximated DGFVM solution.
Now by subtracting $\mathscr{B}(u;u_{h},u_{h})$ from both side of equation (36), we get
\begin{align}
\mathscr{B}(u;e,u_{h})= \sum_{K \in \mathscr{R}_{h}}\sum_{j=1}^{4}\int_{A_{j+1}CA_{j}}
(\epsilon(u_{h})-\epsilon(u))\nabla u_{h}.\bold{n}\gamma v_{h} ds \nonumber \\
+\sum_{e \in \Gamma}\int_{e}[\gamma v_{h}](\epsilon(u_{h})-\epsilon(u))\nabla u_{h}ds \nonumber \\
-\sum_{K \in \mathscr{R}_{h}}\sum_{j=1}^{4}\int_{A_{j+1}CA_{j}}
(\rho(u_{h}h_d(x))-\rho(u)h_d(x))\vec{\beta}.\bold{n}\gamma v_{h} ds \nonumber \\
-\sum_{e \in \Gamma}\int_{e}[\gamma v_{h}]\Big\{\rho(u_{h}h_d(x))-\rho(u)h_d(x)\vec{\beta}.\bold{n}\Big\}ds
\end{align}
Now adding both side in above equation following term
\begin{gather}
\sum_{K \in \mathscr{R}_{h}}\sum_{j=1}^{4}\int_{A_{j+1}CA_{j}}
\epsilon_{u}(u_{h})(u_{h}-u)\nabla u.\bold{n}\gamma v_{h}ds
+\sum_{e \in \Gamma}\int_{e}[\gamma v_{h}]\Big\{\epsilon_{u}(u_{h})(u_{h}-u)\nabla u \Big\}ds \nonumber \\
-\sum_{K \in \mathscr{R}_{h}}\sum_{j=1}^{4}\int_{A_{j+1}CA_{j}}
(\rho h_{d})_{u}(u_{h})(u_{h}-u) \vec{\beta}.\bold{n}\gamma v_{h}ds
-\sum_{e \in \Gamma}\int_{e}[\gamma v_{h}]\{(\rho h_{d})_{u}(u_{h})(u_{h}-u)\}ds.
\end{gather}
Now we split error term as $$e=u-u_{h}=u-{\Pi}_{h}u+{\Pi}_{h}u-u_{h}$$
and using Taylor's formula for linearizion given in ()-() we rewrite equation (42) as
\begin{align}
 \bar{\mathscr{B}}(u;{\Pi}_{h}u-u_{h},v_{h})=\bar{\mathscr{B}}(u;{\Pi}_{h}u-u,v_{h})+\mathscr{F}(u_{h};u_{h}-u,v_{h}),
\end{align}
where
\begin{align}
 \mathscr{F}(u_{h};u_{h}-u,v_{h})=\sum_{K \in \mathscr{R}_{h}}\sum_{j=1}^{4}\int_{A_{j+1}CA_{j}}
\tilde{\epsilon}_{u}(u_{h})e\nabla e.\bold{n}\gamma v_{h}ds \nonumber \\
+\sum_{e \in \Gamma}\int_{e}[\gamma v_{h}]\Big\{\epsilon_{u}(u_{h})e\nabla e \Big\}ds
+\sum_{K \in \mathscr{R}_{h}}\sum_{j=1}^{4}\int_{A_{j+1}CA_{j}}
\tilde{\epsilon}_{uu}(u_{h})e^{2}\nabla u.\bold{n}\gamma v_{h}ds \nonumber \\
-\sum_{K \in \mathscr{R}_{h}}\sum_{j=1}^{4}\int_{A_{j+1}CA_{j}}
\tilde{(\rho h_{d})}_{uu}(u_{h})e^{2}\vec{\beta}.\bold{n}\gamma v_{h}ds
-\sum_{e \in \Gamma}\int_{e}[\gamma v_{h}]\Big\{(\rho h_{d})_{uu}(u_{h}) e^{2} \Big\}ds
\end{align}
Note that solving (35) is equivalent to solving (45).
Now for showing there exist at least one $u_{h} \in \mathscr{V}_{h}$ solution to the above equation (45) we consider a map 
$$\mathcal{S}:\mathscr{V}_{h}\rightarrow \mathscr{V}_{h}$$ defined as 
$\mathcal{S}(u_{\varphi}) =\varphi \in  \mathscr{V}_{h}, \quad \forall u_{\varphi} \in \mathscr{V}_{h}$ such that
\begin{align}
 \bar{\mathscr{B}}(u;{\Pi}_{h}u-\varphi,v_{h})=\bar{\mathscr{B}}(u;{\Pi}_{h}u-u,v_{h})+\mathscr{F}(u_{\varphi};u_{\varphi}-u,v_{h}) 
\end{align}
holds.
Consider the closed neighborhood $\mathcal{Q}_{\delta}(\Pi_{h}u)$ of the diameter $\delta >0$.
$$\mathcal{Q}_{\delta}(\Pi_{h}u)=\Big\{ u_{\varphi} \in \mathscr{V}_{h}: \vertiii{u_{\varphi}-\Pi_{h}u} \le \delta \Big\}. $$
Now we first show that $\mathcal{S}$ map closed neighborhood $\mathcal{Q}_{\delta}(\Pi_{h}u)$ into itself and then prove 
existence of DGFVM solution by exploiting Browder's fixed point theorem. The proof can be break using following lemmas.
\begin{lemma}
Let $u_{\varphi}, v_{h} \in \mathscr{V}_{h}$ also set $\chi=u_{\varphi}-\Pi_{h}u$ and $\eta=u-\Pi_{h}u$. Then there exists a constant
$C \ge 0$ (independent of $h$) such that
\begin{gather}
 |\mathscr{F}(u_{\varphi};u_{\varphi}-u,v_{h})| \le C_{\epsilon}\Big[\vertiii{\chi}^{2}
 +C_{u}(h^{5/3}+h^{1/2}+h+h^{2/3}+h^{3/2})\vertiii{\chi}\nonumber \\
 +C_{u}(h^{2}+h+h^{3/2})\vertiii{\eta} \Big]\vertiii{v_{h}}+
 C_{\rho h_{d}}\Big[\vertiii{\chi}^{2}+C_{u}(h^{5/3}+h^{3/2})\vertiii{\chi} \nonumber \\
 + C_{u}(h^{3/2}+h)\vertiii{\eta}\Big]\vertiii{v_{h}}
.
\end{gather}
\end{lemma}
\begin{proof}
 Let $u_{\varphi} \in \mathscr{V}_{h}$ and take $\zeta=u_{\varphi}-u$ in equation (45) we write
 $u_{\varphi}$ in place of $u_{h}$ and $\zeta=u_{\varphi}-u$ to get
 \begin{gather}
 \mathscr{F}(u_{\varphi};\zeta,v_{h})=\sum_{K \in \mathscr{R}_{h}}\sum_{j=1}^{4}\int_{A_{j+1}CA_{j}}
\tilde{\epsilon}_{u}(u_{\varphi})\zeta\nabla \zeta.\bold{n}\gamma v_{h}ds \nonumber \\
+\sum_{e \in \Gamma}\int_{e}[\gamma v_{h}]\Big\{\epsilon_{u}(u_{\varphi})\zeta\nabla \zeta \Big\}ds
+\sum_{K \in \mathscr{R}_{h}}\sum_{j=1}^{4}\int_{A_{j+1}CA_{j}}
\tilde{\epsilon}_{uu}(u_{\varphi})\zeta^{2}\nabla u.\bold{n}\gamma v_{h}ds \nonumber \\
-\sum_{K \in \mathscr{R}_{h}}\sum_{j=1}^{4}\int_{A_{j+1}CA_{j}}
\tilde{(\rho h_{d})}_{uu}(u_{\varphi})\zeta^{2}\vec{\beta}.\bold{n}\gamma v_{h}ds
-\sum_{e \in \Gamma}\int_{e}[\gamma v_{h}]\Big\{(\rho h_{d})_{uu}(u_{\varphi}) \zeta^{2} \Big\}ds. 
 \end{gather}
Now split $\zeta=\chi-\eta$ where $\chi=u_{\varphi}-\Pi_{h}u$ and $\eta=u-\Pi_{h}u$. Then right hand side is estimated in 
following way. The First term is estimated as
\begin{gather}
 \Big|\sum_{K \in \mathscr{R}_{h}}\sum_{j=1}^{4}\int_{A_{j+1}CA_{j}}
\tilde{\epsilon}_{u}(u_{\varphi})\zeta\nabla \zeta.\bold{n}\gamma v_{h}ds \Big| \le 
 \Big|\sum_{K \in \mathscr{R}_{h}}\sum_{j=1}^{4}\int_{A_{j+1}CA_{j}}
\tilde{\epsilon}_{u}(u_{\varphi})\chi\nabla \chi.\bold{n}\gamma v_{h}ds \Big|\nonumber \\
+\Big|\sum_{K \in \mathscr{R}_{h}}\sum_{j=1}^{4}\int_{A_{j+1}CA_{j}}
\tilde{\epsilon}_{u}(u_{\varphi})\chi\nabla \eta.\bold{n}\gamma v_{h}ds\Big|+
 \Big|\sum_{K \in \mathscr{R}_{h}}\sum_{j=1}^{4}\int_{A_{j+1}CA_{j}}
\tilde{\epsilon}_{u}(u_{\varphi})\eta\nabla \chi.\bold{n}\gamma v_{h}ds\Big|\nonumber \\
+\Big|\sum_{K \in \mathscr{R}_{h}}\sum_{j=1}^{4}\int_{A_{j+1}CA_{j}}
\tilde{\epsilon}_{u}(u_{\varphi})\eta\nabla \eta.\bold{n}\gamma v_{h}ds\Big|.
\end{gather}
Second term is estimated as
\begin{gather}
\Big|\sum_{e \in \Gamma}\int_{e}[\gamma v_{h}]\Big\{\epsilon_{u}(u_{\varphi})\zeta\nabla \zeta \Big\}ds\Big|\le
\Big|\sum_{e \in \Gamma}\int_{e}[\gamma v_{h}]\Big\{\epsilon_{u}(u_{\varphi})\chi \nabla \chi \Big\}ds\Big| \nonumber \\
+\Big|\sum_{e \in \Gamma}\int_{e}[\gamma v_{h}]\Big\{\epsilon_{u}(u_{\varphi})\eta\nabla \chi \Big\}ds\Big|
+\Big|\sum_{e \in \Gamma}\int_{e}[\gamma v_{h}]\Big\{\epsilon_{u}(u_{\varphi})\chi \nabla \eta \Big\}ds\Big|\nonumber \\
+\Big|\sum_{e \in \Gamma}\int_{e}[\gamma v_{h}]\Big\{\epsilon_{u}(u_{\varphi})\eta\nabla \eta \Big\}ds\Big|.
\end{gather}
Third term is estimated as
\begin{gather}
\Big|\sum_{K \in \mathscr{R}_{h}}\sum_{j=1}^{4}\int_{A_{j+1}CA_{j}}
\tilde{\epsilon}_{uu}(u_{\varphi})\zeta^{2}\nabla u.\bold{n}\gamma v_{h}ds\Big| \le
\Big|\sum_{K \in \mathscr{R}_{h}}\sum_{j=1}^{4}\int_{A_{j+1}CA_{j}}
\tilde{\epsilon}_{uu}(u_{\varphi})\eta^{2}\nabla u.\bold{n}\gamma v_{h}ds\Big| \nonumber \\
+2\Big|\sum_{K \in \mathscr{R}_{h}}\sum_{j=1}^{4}\int_{A_{j+1}CA_{j}}
\tilde{\epsilon}_{uu}(u_{\varphi})\eta.\chi \nabla u.\bold{n}\gamma v_{h}ds\Big|+
\Big|\sum_{K \in \mathscr{R}_{h}}\sum_{j=1}^{4}\int_{A_{j+1}CA_{j}}
\tilde{\epsilon}_{uu}(u_{\varphi})\chi^{2}\nabla u.\bold{n}\gamma v_{h}ds\Big|.
\end{gather}
Fourth term is estimated as
\begin{gather}
\Big|\sum_{K \in \mathscr{R}_{h}}\sum_{j=1}^{4}\int_{A_{j+1}CA_{j}}
\tilde{(\rho h_{d})}_{uu}(u_{\varphi})\zeta^{2}\vec{\beta}.\bold{n}\gamma v_{h}ds\Big| \le
\Big|\sum_{K \in \mathscr{R}_{h}}\sum_{j=1}^{4}\int_{A_{j+1}CA_{j}}
\tilde{(\rho h_{d})}_{uu}(u_{\varphi})\chi^{2}\vec{\beta}.\bold{n}\gamma v_{h}ds\Big|\nonumber \\
+2\Big|\sum_{K \in \mathscr{R}_{h}}\sum_{j=1}^{4}\int_{A_{j+1}CA_{j}}
\tilde{(\rho h_{d})}_{uu}(u_{\varphi})\eta.\chi\vec{\beta}.\bold{n}\gamma v_{h}ds\Big|
+\Big|\sum_{K \in \mathscr{R}_{h}}\sum_{j=1}^{4}\int_{A_{j+1}CA_{j}}
\tilde{(\rho h_{d})}_{uu}(u_{\varphi})\eta^{2}\vec{\beta}.\bold{n}\gamma v_{h}ds\Big|.
\end{gather}
Fifth term is estimated as
\begin{gather}
 \Big|\sum_{e \in \Gamma}\int_{e}[\gamma v_{h}]\Big\{(\rho h_{d})_{uu}(u_{\varphi}) \zeta^{2} \Big\}ds\Big| \le
 \Big|\sum_{e \in \Gamma}\int_{e}[\gamma v_{h}]\Big\{(\rho h_{d})_{uu}(u_{\varphi}) \chi^{2} \Big\}ds\Big|\nonumber \\
 +2\Big|\sum_{e \in \Gamma}\int_{e}[\gamma v_{h}]\Big\{(\rho h_{d})_{uu}(u_{\varphi}) \eta.\chi \Big\}ds\Big|+
 \Big|\sum_{e \in \Gamma}\int_{e}[\gamma v_{h}]\Big\{(\rho h_{d})_{uu}(u_{\varphi}) \eta^{2} \Big\}ds\Big|.
\end{gather}
In equation (49) first term is estimated as
\begin{gather}
 \Big|\sum_{K \in \mathscr{R}_{h}}\sum_{j=1}^{4}\int_{A_{j+1}CA_{j}}
\tilde{\epsilon}_{u}(u_{\varphi})\chi\nabla \chi.\bold{n}\gamma v_{h}ds \Big| \le
\Big|\sum_{K}\langle {\epsilon}_{u}(u_{\varphi})\chi \nabla \chi,\nabla v_{h} \rangle \Big| \nonumber \\
+\Big|\sum_{K}\int_{\partial K}[\gamma v_{h}-v_{h}]\Big\{{\epsilon}_{u}(u_{\varphi})\chi \nabla \chi.\bold{n}\Big\}ds\Big|
+\Big|\sum_{K}\langle \nabla\Big( {\epsilon}_{u}(u_{\varphi})\chi \nabla \chi\Big), v_{h}-\gamma v_{h} \rangle\Big|.
\end{gather}
First part of equation (54) is estimated as
\begin{gather*}
\Big|\sum_{K}\langle {\epsilon}_{u}(u_{\varphi})\chi \nabla \chi,\nabla v_{h} \rangle \Big| \le
C_{\epsilon}\sum_{K}\int_{K}|\chi.\nabla \chi. \nabla v_{h}|dx.
\end{gather*}
Now using holder's inequality we get
\begin{gather}
C_{\epsilon}\sum_{K}\int_{K}|\chi.\nabla \chi. \nabla v_{h}|dx \le 
C_{\epsilon}\sum_{K}\lVert \chi \rVert_{L^{6}(K)}\lVert \chi \rVert_{L^{3}(K)}\lVert \nabla v_{h} \rVert_{L^{2}(K)} \nonumber \\
\le C_{\epsilon} \vertiii{\chi}\vertiii{\chi}\vertiii{v_{h}}.
\end{gather}
Now second part of equation (54) is estimated using Holder's inequality and trace inequality
\[\Big|\sum_{K}\int_{\partial K}[\gamma v_{h}-v_{h}]\Big\{{\epsilon}_{u}(u_{\varphi})\chi \nabla \chi.\bold{n}\Big\}ds\Big|
\le\] 
\[C_{\epsilon}\sum_{K}\Big(\int_{\partial K}[\gamma v_{h}-v_{h}]^{2}\Big)^{1/2}\lVert \chi \rVert_{L^{4}(\partial K)}
\lVert \nabla \chi \rVert_{L^{4}(\partial K)}.\]
Now using {\bf trace inequality} defined as
\begin{gather}
\lVert \nabla \chi \rVert_{L^{4}(\partial K)} \le C_{h}\Big(h^{-1}\lVert \nabla \chi \rVert^{4}_{L^{4}(K)}
+\lVert \nabla \chi \rVert^{3}_{L^{6}(K)}\lVert \nabla.\nabla \chi \rVert_{L^{2}(K)}\Big) \\
\lVert \chi \rVert_{L^{4}(\partial K)} \le C_{h}\Big(h^{-1}\lVert \chi \rVert^{4}_{L^{4}(K)}
+\lVert \chi \rVert^{3}_{L^{6}(K)}\lVert \chi \rVert_{L^{2}(K)}\Big)
\end{gather}
we get that 
\[\le C_{\epsilon}
\Big(h^{-1}\vert \gamma v_{h}-v_{h} \vert^{2}_{L^{2}(K)}+h\vert \gamma v_{h}-v_{h} \vert^{2}_{H^{1}(K)}\Big)^{1/2} \times\]
\[\Big(h^{-1}\lVert \chi \rVert^{4}_{L^{4}(K)}
+\lVert \chi \rVert^{3}_{L^{6}(K)}\lVert \chi \rVert_{L^{2}(K)}\Big)^{1/4}\]
\[\times \Big(h^{-1}\lVert \nabla \chi \rVert^{4}_{L^{4}(K)}
+\lVert \nabla \chi \rVert^{3}_{L^{6}(K)}\lVert \nabla.\nabla \chi \rVert_{L^{2}(K)}\Big)^{1/4}\]
\begin{gather}
 \le C_{\epsilon} \vertiii{\chi}\vertiii{\chi}\vertiii{v_{h}}
\end{gather}
Third term of equation (54) is estimated in similar way and it is written as
\begin{gather}
\Big|\sum_{K}\langle \nabla\Big( {\epsilon}_{u}(u_{\varphi})\chi \nabla \chi\Big), v_{h}-\gamma v_{h} \rangle\Big| \le
C_{\epsilon} \vertiii{\chi}\vertiii{\chi}\vertiii{v_{h}}.
\end{gather}
Now second term equation (49) is estimated as 
\begin{gather}
 \Big|\sum_{K \in \mathscr{R}_{h}}\sum_{j=1}^{4}\int_{A_{j+1}CA_{j}}
\tilde{\epsilon}_{u}(u_{\varphi})\chi\nabla \eta.\bold{n}\gamma v_{h}ds \Big| \le
\Big|\sum_{K}\langle {\epsilon}_{u}(u_{\varphi})\chi \nabla \eta,\nabla v_{h} \rangle \Big| \nonumber \\
+\Big|\sum_{K}\int_{\partial K}[\gamma v_{h}-v_{h}]\Big\{{\epsilon}_{u}(u_{\varphi})\chi \nabla \eta.\bold{n}\Big\}ds\Big|
+\Big|\sum_{K}\langle \nabla\Big( {\epsilon}_{u}(u_{\varphi})\chi \nabla \eta\Big), v_{h}-\gamma v_{h} \rangle\Big|.
\end{gather}
Now first term of equation (60) is estimated using Holder's inequality as
\begin{gather}
\Big|\sum_{K}\langle {\epsilon}_{u}(u_{\varphi})\chi \nabla \eta,\nabla v_{h} \rangle \Big| \le
C_{\epsilon}\sum_{K}\int_{K}|\chi.\nabla \eta. \nabla v_{h}|dx \nonumber \\
\le C_{\epsilon}\sum_{K}\Vert \chi \Vert_{L^{6}(K)} \Vert \nabla \eta \Vert_{L^{3}(K)}\Vert \nabla v_{h} \Vert_{L^{2}(K)}
\end{gather}
Now using {\bf inverse inequality} defined as
\begin{gather}
 \Vert v_{h} \Vert_{L^{r}(K)} \le C h^{2/r-1}\Vert v_{h}  \Vert_{L^{2}(K)} \quad \forall r \ge 2.
\end{gather}
and also using approximation property we get
\begin{gather}
 \le C_{\epsilon}C_{u}h^{-1/3}\Vert \nabla \eta \Vert_{L^{2}(K)}\vertiii{\chi}\vertiii{v_{h}} \nonumber \\
 \le C_{\epsilon}C_{u}h^{2/3}\Vert u \Vert_{H^{2}(\Omega)}\vertiii{\chi}\vertiii{v_{h}}.
\end{gather}
Second term of equation (60) is estimated as using Holder's inequality and trace inequality  
\begin{gather}
 \Big|\sum_{K}\int_{\partial K}[\gamma v_{h}-v_{h}]\Big\{{\epsilon}_{u}(u_{\varphi})\chi \nabla \eta.\bold{n}\Big\}ds\Big|
 \le 
 C_{\epsilon}\sum_{K}\Big( \int_{\partial K}[\gamma v_{h}-v_{h}]^{2}\Big)^{1/2}\Vert \chi \Vert_{L^{4}(\partial K)} 
 \Vert \nabla \eta \Vert_{L^{4}(\partial K)}\nonumber \\
 \le C_{\epsilon}h^{-1}
 \sum_{K}\Big(|\gamma v_{h}-v_{h}|^{2}_{L^{2}(K)}+h^{2}|\gamma v_{h}-v_{h}|^{2}_{H^{1}(K)} \Big)^{1/2} \times
 \Big(\Vert \chi \Vert^{4}_{L^{4}(K)}+h\Vert \chi \Vert^{3}_{L^{6}(K)} \Vert \nabla \chi \Vert_{L^{2}(K)} \Big)^{1/4} \nonumber \\
 \times \Big(\Vert \nabla \eta \Vert^{4}_{L^{4}(K)}+ h\Vert \nabla \eta \Vert^{3}_{L^{6}(K)}
 \Vert \nabla.\nabla \eta \Vert_{L^{2}(K)}\Big)^{1/4} \nonumber \\
 \le C_{\epsilon}h^{1/2}\Vert u \Vert_{H^{2}(\Omega)}\vertiii{v_{h}}\vertiii{\chi}.
\end{gather}
Third term of equation (60) is estimated as 
\begin{gather}
 \Big|\sum_{K}\langle \nabla\Big( {\epsilon}_{u}(u_{\varphi})\chi \nabla \eta\Big), v_{h}-\gamma v_{h} \rangle\Big|
 \le C_{\epsilon}C_{u}(h^{2/3}\Vert u \Vert_{H^{2}(K)}\vertiii{\chi}\vertiii{v_{h}} \nonumber \\
 +h^{1/2}\Vert u \Vert_{H^{2}(\Omega)}\vertiii{v_{h}}\vertiii{\chi}).
\end{gather}
Now third term of equation (49) is estimated as
\begin{gather}
 \Big|\sum_{K \in \mathscr{R}_{h}}\sum_{j=1}^{4}\int_{A_{j+1}CA_{j}}
\tilde{\epsilon}_{u}(u_{\varphi})\eta\nabla \chi.\bold{n}\gamma v_{h}ds \Big| \le
\Big|\sum_{K}\langle {\epsilon}_{u}(u_{\varphi})\eta \nabla \chi,\nabla v_{h} \rangle \Big| \nonumber \\
+\Big|\sum_{K}\int_{\partial K}[\gamma v_{h}-v_{h}]\Big\{{\epsilon}_{u}(u_{\varphi})\eta \nabla \chi.\bold{n}\Big\}ds\Big|
+\Big|\sum_{K}\langle \nabla\Big( {\epsilon}_{u}(u_{\varphi})\eta \nabla \chi\Big), v_{h}-\gamma v_{h} \rangle\Big|.
\end{gather}
First part of equation (66) is estimated by using Holder's inequality as
\begin{gather}
 \Big|\sum_{K}\langle {\epsilon}_{u}(u_{\varphi})\eta \nabla \chi,\nabla v_{h} \rangle \Big| \le 
 C_{\epsilon}\sum_{K}\Vert \eta \Vert_{L^{6}(K)}\Vert \nabla \chi \Vert_{L^{3}(K)}\Vert \nabla v_{h} \Vert_{L^{2}(K)} \nonumber \\
\le 
C_{\epsilon}\sum_{K}h^{2/6-1}\Vert \eta \Vert_{L^{2}(K)}h^{2/3-1}\Vert \nabla \chi \Vert_{L^{2}(K)}
\Vert \nabla v_{h} \Vert_{L^{2}(K)}\nonumber \\
\le C_{\epsilon}C_{u}h\Vert u \Vert_{H^{2}(\Omega)}\vertiii{\chi}\vertiii{v_{h}}.
\end{gather}
Second part of equation (66) is estimated using trace inequality we have
\begin{gather}
 \Big|\sum_{K}\int_{\partial K}[\gamma v_{h}-v_{h}]\Big\{{\epsilon}_{u}(u_{\varphi})\eta \nabla \chi.\bold{n}\Big\}ds\Big|
 \le
 C_{\epsilon}\sum_{K}\Big(\int_{\partial K}[\gamma v_{h}-v_{h}]^{2}ds\Big)^{1/2}\Vert \eta \Vert_{L^{4}(\partial K)} 
 \Vert \nabla \chi\Vert_{L^{4}(\partial K)} \nonumber \\
 \le
 C_{\epsilon}h^{3/2}\Vert u \Vert_{H^{2}(\Omega)}\vertiii{\chi}\vertiii{v_{h}}
\end{gather}
Third part of equation (66) is estimated as 
\begin{gather}
 \Big|\sum_{K}\langle \nabla\Big( {\epsilon}_{u}(u_{\varphi})\eta \nabla \chi\Big), v_{h}-\gamma v_{h} \rangle\Big|
 \le
 C_{\epsilon}C_{u}h\Vert u \Vert_{H^{2}(\Omega)}\vertiii{\chi}\vertiii{v_{h}}
 +C_{\epsilon}h^{3/2}\Vert u \Vert_{H^{2}(\Omega)}\vertiii{\chi}\vertiii{v_{h}}
\end{gather}
Fourth term of equation (49) is estimated as
\begin{gather}
 \Big|\sum_{K \in \mathscr{R}_{h}}\sum_{j=1}^{4}\int_{A_{j+1}CA_{j}}
\tilde{\epsilon}_{u}(u_{\varphi})\eta\nabla \eta.\bold{n}\gamma v_{h}ds \Big| \le
\Big|\sum_{K}\langle {\epsilon}_{u}(u_{\varphi})\eta \nabla \eta,\nabla v_{h} \rangle \Big| \nonumber \\
+\Big|\sum_{K}\int_{\partial K}[\gamma v_{h}-v_{h}]\Big\{{\epsilon}_{u}(u_{\varphi})\eta \nabla \eta.\bold{n}\Big\}ds\Big|
+\Big|\sum_{K}\langle \nabla\Big( {\epsilon}_{u}(u_{\varphi})\eta \nabla \eta\Big), v_{h}-\gamma v_{h} \rangle\Big|.
\end{gather}
First part of equation (70) is estimated using Holder's inequality as
\begin{gather}
\Big|\sum_{K}\langle {\epsilon}_{u}(u_{\varphi})\eta \nabla \eta,\nabla v_{h} \rangle \Big| 
\le
C_{\epsilon}\sum_{K}\Vert \eta \Vert_{L^{6}(K)} \Vert \nabla \eta \Vert_{L^{3}(K)}\Vert \nabla v_{h} \Vert_{L^{2}(K)} \nonumber \\
\le
C_{\epsilon}C_{u} h \Vert u \Vert_{H^{2}(\Omega)} \vertiii{\eta} \vertiii{v_{h}}
\end{gather}
Second part of equation (70) is estimated as
\begin{gather}
 \Big|\sum_{K}\int_{\partial K}[\gamma v_{h}-v_{h}]\Big\{{\epsilon}_{u}(u_{\varphi})\eta \nabla \eta.\bold{n}\Big\}ds\Big|
 \le
C_{\epsilon}\sum_{K}\Big(\int_{\partial K}[\gamma v_{h}-v_{h}]^{2}ds\Big)^{1/2}\Vert \eta \Vert_{L^{4}(\partial K)} 
 \Vert \nabla \eta \Vert_{L^{4}(\partial K)} \nonumber \\
\le C_{\epsilon}h^{3/2}\Vert u \Vert_{H^{2}(\Omega)}\vertiii{\eta}\vertiii{v_{h}}
\end{gather}
Third part of equation (70) is estimated as
\begin{gather}
\Big|\sum_{K}\langle \nabla\Big( {\epsilon}_{u}(u_{\varphi})\eta \nabla \eta\Big), v_{h}-\gamma v_{h} \rangle\Big|
\le C_{\epsilon}C_{u} h \Vert u \Vert_{H^{2}(\Omega)} \vertiii{\eta} \vertiii{v_{h}}
+C_{\epsilon}h^{3/2}\Vert u \Vert_{H^{2}(\Omega)}\vertiii{\eta}\vertiii{v_{h}}.
\end{gather}
Now first part of equation (50) is estimated as
\begin{gather}
 \Big|\sum_{e \in \Gamma}\int_{e}[\gamma v_{h}]\Big\{\epsilon_{u}(u_{\varphi})\chi \nabla \chi \Big\}ds\Big|
 \le
 C_{\epsilon}\sum_{K}\Big([\gamma v_{h}]^{2}\Big)^{1/2}
 \Vert \chi \Vert_{L^{4}(\partial K)}\Vert \nabla \chi \Vert_{L^{4}(\partial K)} \nonumber \\
 \le
 C_{\epsilon}\sum_{K}\Big([\gamma v_{h}]^{2}\Big)^{1/2}
 \Big(\Vert \chi \Vert^{4}_{L^{4}(K)}+h\Vert \chi \Vert^{3}_{L^{6}(K)} \Vert \nabla \chi \Vert_{L^{2}(K)} \Big)^{1/4}\nonumber \\
 \times \Big(\Vert \nabla \chi \Vert^{4}_{L^{4}(K)}
 +h\Vert \nabla \chi \Vert^{3}_{L^{6}(K)} \Vert \nabla .\nabla \chi \Vert_{L^{2}(K)} \Big)^{1/4} \nonumber \\
 \le C_{\epsilon}\vertiii{v_{h}}\vertiii{\chi}\vertiii{\chi}
\end{gather}
In similar way we can show that second, third and fourth part of equation (50) is estimated as
\begin{gather}
\Big|\sum_{e \in \Gamma}\int_{e}[\gamma v_{h}]\Big\{\epsilon_{u}(u_{\varphi})\chi \nabla \eta \Big\}ds\Big|
\le 
C_{\epsilon}C_{u}h^{1/2}\Vert u \Vert_{H^{2}(\Omega)}\vertiii{v_{h}}\vertiii{\chi} \\
\Big|\sum_{e \in \Gamma}\int_{e}[\gamma v_{h}]\Big\{\epsilon_{u}(u_{\varphi})\eta \nabla \chi \Big\}ds\Big|
\le 
C_{\epsilon}C_{u}h^{3/2}\Vert u \Vert_{H^{2}(\Omega)}\vertiii{v_{h}}\vertiii{\chi} \\
\Big|\sum_{e \in \Gamma}\int_{e}[\gamma v_{h}]\Big\{\epsilon_{u}(u_{\varphi})\eta \nabla \eta \Big\}ds\Big|
\le 
C_{\epsilon}C_{u}h^{3/2}\Vert u \Vert_{H^{2}(\Omega)}\vertiii{v_{h}}\vertiii{\eta}.
\end{gather}
First part of equation (51) is estimated using similar argument as
\begin{gather}
 \Big|\sum_{K \in \mathscr{R}_{h}}\sum_{j=1}^{4}\int_{A_{j+1}CA_{j}}
\tilde{\epsilon}_{uu}(u_{\varphi})\chi^{2}\nabla u.\bold{n}\gamma v_{h}ds\Big| \le
C_{\epsilon}C_{u}\vertiii{\chi}^{2}\vertiii{v_{h}}
\end{gather}
Second part of equation (51) is estimated using similar argument as
\begin{gather}
 \Big|\sum_{K \in \mathscr{R}_{h}}\sum_{j=1}^{4}\int_{A_{j+1}CA_{j}}
\tilde{\epsilon}_{uu}(u_{\varphi})\chi.\eta\nabla u.\bold{n}\gamma v_{h}ds\Big| \le
C_{\epsilon}C_{u}\Big(h^{5/3}\Vert u \Vert_{H^{2}(\Omega)}\vertiii{\chi}\vertiii{v_{h}}+ \nonumber \\
h^{3/2}\Vert u \Vert_{H^{2}(\Omega)}\vertiii{\chi}\vertiii{v_{h}}\Big)
\end{gather}
Third part of equation (51) is estimated using similar argument as
\begin{gather}
 \Big|\sum_{K \in \mathscr{R}_{h}}\sum_{j=1}^{4}\int_{A_{j+1}CA_{j}}
\tilde{\epsilon}_{uu}(u_{\varphi})\eta^{2}\nabla u.\bold{n}\gamma v_{h}ds\Big| \le
C_{\epsilon}C_{u}\Big(h^{2}\Vert u \Vert_{H^{2}(\Omega)}\vertiii{\eta}\vertiii{v_{h}} \nonumber \\
+h^{3/2}\Vert u \Vert_{H^{2}(\Omega)}\vertiii{\eta}\vertiii{v_{h}}\Big).
\end{gather}
First part of equation (52) is estimated as
\begin{gather}
\Big|\sum_{K \in \mathscr{R}_{h}}\sum_{j=1}^{4}\int_{A_{j+1}CA_{j}}
\tilde{(\rho h_{d})}_{uu}(u_{\varphi})\chi^{2}\vec{\beta}.\bold{n}\gamma v_{h}ds\Big|\le
C_{\rho h_{d}}\vertiii{v_{h}}\vertiii{\chi}^{2}.
\end{gather}
Second part of equation (52) is estimated as
\begin{gather}
\Big|\sum_{K \in \mathscr{R}_{h}}\sum_{j=1}^{4}\int_{A_{j+1}CA_{j}}
\tilde{(\rho h_{d})}_{uu}(u_{\varphi})\eta.\chi\vec{\beta}.\bold{n}\gamma v_{h}ds\Big|\le
C_{\rho h_{d}}\Big(h^{5/3}\Vert u \Vert_{H^{2}(\Omega)}\vertiii{\chi}\vertiii{v_{h}} \nonumber \\
+h^{3/2}\Vert u \Vert_{H^{2}(\Omega)}\vertiii{\chi}\vertiii{v_{h}}\Big).
\end{gather}
Third part of equation (52) is estimated as
\begin{gather}
 \Big|\sum_{K \in \mathscr{R}_{h}}\sum_{j=1}^{4}\int_{A_{j+1}CA_{j}}
\tilde{(\rho h_{d})}_{uu}(u_{\varphi})\eta^{2}\vec{\beta}.\bold{n}\gamma v_{h}ds\Big|\le
C_{\rho h_{d}}h\Vert u \Vert_{H^{2}(\Omega)}\vertiii{\eta}\vertiii{v_{h}} \nonumber \\
+C_{\rho h_{d}}h^{3/2}\Vert u \Vert_{H^{2}(\Omega)}\vertiii{\eta}\vertiii{v_{h}}
\end{gather}
Now equation (53) is estimated using similar argument as
\begin{gather}
 \Big|\sum_{e \in \Gamma}\int_{e}[\gamma v_{h}]\Big\{(\rho h_{d})_{uu}(u_{\varphi}) \zeta^{2} \Big\}ds\Big| \le
 \Big|\sum_{e \in \Gamma}\int_{e}[\gamma v_{h}]\Big\{(\rho h_{d})_{uu}(u_{\varphi}) \chi^{2} \Big\}ds\Big|\nonumber \\
 +2\Big|\sum_{e \in \Gamma}\int_{e}[\gamma v_{h}]\Big\{(\rho h_{d})_{uu}(u_{\varphi}) \eta.\chi \Big\}ds\Big|+
 \Big|\sum_{e \in \Gamma}\int_{e}[\gamma v_{h}]\Big\{(\rho h_{d})_{uu}(u_{\varphi}) \eta^{2} \Big\}ds\Big|\nonumber \\
 \le
 C_{\rho h_{d}}\Big(\vertiii{v_{h}}\vertiii{\chi}^{2}
 +h^{3/2}\Vert u \Vert_{H^{2}(\Omega)}\vertiii{\chi}\vertiii{v_{h}}
 +h^{3/2}\Vert u \Vert_{H^{2}(\Omega)}\vertiii{\eta}\vertiii{v_{h}}\Big).
\end{gather}

\end{proof}
Now we are interested in deriving upper bound of $\vertiii{\Pi_{h}u-\varphi}$ and it is explained in next lemma.
\begin{lemma}
Let $u_{\varphi} \in \mathscr{V}_{h}$ and take $\varphi=\mathcal{S}u_{\varphi}$.Then there exist a positive constant $C$
(independent of $h$) such that
\begin{gather}
 \vertiii{\Pi_{h}u-\varphi} \le 
 C_{\epsilon}\Big[\vertiii{\Pi_{h}u-u_{\varphi}}^{2}
 +C_{u}\Big(h^{5/3}+h^{1/2}+h^{2/3}+h(1+h^{1/2})\Big) \nonumber \\
 \vertiii{\Pi_{h}u-u_{\varphi}}
 +C_{u}(h^{2}+h+h^{3/2})\vertiii{\eta} \Big]+
 C_{\rho h_{d}}\Big[\vertiii{\Pi_{h}u-u_{\varphi}}^{2}+C_{u}(h^{5/3}+h^{3/2}) \nonumber \\
 \vertiii{\Pi_{h}u-u_{\varphi}} + C_{u}(h^{3/2}+h)\vertiii{\eta}\Big]+C\vertiii{\eta}.
\end{gather}
holds.
\end{lemma}
\begin{proof}
In equation (46) we redefine the term $\chi=\Pi_{h}u-u_{\varphi}$, $\eta=\Pi_{h}u-u$ , and $\vartheta=\Pi_{h}u-\varphi$.
Now consider the first term in the right hand side of equation (46) and replace $v_{h}=\vartheta$ and use the boundedness 
property of the operator to get
\begin{align}
 \Big| \bar{\mathscr{B}}(u;\eta,\vartheta)\Big|\le C\vertiii{\eta}_{\nu}\vertiii{\vartheta}
\end{align}
Also by replacing $v_{h}=\vartheta$ in previous lemma 3.8 we obtain
\begin{gather}
 \Big|\mathscr{F}(u_{\varphi};u_{\varphi}-u,\vartheta)  \Big| \le 
 C_{\epsilon}\Big[\vertiii{\chi}^{2}
 +C_{u}\Big(h^{5/3}+h^{1/2}+h^{2/3}+h(1+h^{1/2})\Big) \nonumber \\
 \vertiii{\chi}
 +C_{u}(h^{2}+h+h^{3/2})\vertiii{\eta} \Big]\vertiii{\vartheta}+
 C_{\rho h_{d}}\Big[\vertiii{\chi}^{2}+C_{u}(h^{5/3}+h^{3/2}) \nonumber \\
 \vertiii{\chi} + C_{u}(h^{3/2}+h)\vertiii{\eta}\Big]\vertiii{\vartheta}.
\end{gather}
Now putting the value of equation (86) and (87) in equation (46) we get
\begin{gather}
 \bar{\mathscr{B}}(u;\vartheta,\vartheta) \le
 C_{\epsilon}\Big[\vertiii{\chi}^{2}
 +C_{u}\Big(h^{5/3}+h^{1/2}+h^{2/3}+h(1+h^{1/2})\Big) \nonumber \\
 \vertiii{\chi}
 +C_{u}(h^{2}+h+h^{3/2})\vertiii{\eta} \Big]\vertiii{\vartheta}+
 C_{\rho h_{d}}\Big[\vertiii{\chi}^{2}+C_{u}(h^{5/3}+h^{3/2}) \nonumber \\
 \vertiii{\chi} + C_{u}(h^{3/2}+h)\vertiii{\eta}\Big]\vertiii{\vartheta}
 +C\vertiii{\eta}_{\nu}\vertiii{\vartheta}.
\end{gather}
Now using coercive property we obtain
\begin{gather}
 \vertiii{\vartheta}^{2} \le 
 C_{\epsilon}\Big[\vertiii{\chi}^{2}
 +C_{u}\Big(h^{5/3}+h^{1/2}+h^{2/3}+h(1+h^{1/2})\Big) \nonumber \\
 \vertiii{\chi}
 +C_{u}(h^{2}+h+h^{3/2})\vertiii{\eta} \Big]\vertiii{\vartheta}+
 C_{\rho h_{d}}\Big[\vertiii{\chi}^{2}+C_{u}(h^{5/3}+h^{3/2}) \nonumber \\
 \vertiii{\chi} + C_{u}(h^{3/2}+h)\vertiii{\eta}\Big]\vertiii{\vartheta}
 +C\vertiii{\eta}_{\nu}\vertiii{\vartheta}. 
\end{gather}
Now eliminating $\vartheta$ from both sides we get the desire result.
\end{proof}
\begin{theorem}
 For sufficiently small $h$ there is a $\delta > 0$ such that the map $\mathcal{S}$ maps $\mathcal{Q}_{\delta}(\Pi_{h}u)$ into itself.
\end{theorem}
\begin{proof}
 Let $u_{\varphi} \in \mathcal{Q}(\Pi_{h}u)$ and consider an element $y$ such that $y=\mathcal{S}u_{\varphi}$. Furthermore,
 choose $\delta=h^{-\delta_{0}}\vertiii{\Pi_{h}u-u}$, where $0< \delta_{0} \le 1/4$. Then we get
 \begin{gather}
  \vertiii{\Pi_{h}u-u_{\varphi}}^{2} \le \delta^{2} \nonumber \\
 \vertiii{\Pi_{h}u-u_{\varphi}}^{2} \le h^{-\delta_{0}}\vertiii{\Pi_{h}u-u} \delta \nonumber \\
  \vertiii{\Pi_{h}u-u_{\varphi}}^{2} \le h^{1-\delta_{0}}C\Vert u\Vert_{H^{2}(\Omega)} \delta \nonumber \\
  \vertiii{\Pi_{h}u-u_{\varphi}}^{2} \le h^{1-\delta_{0}}C'_{u}C_{1}\delta.
 \end{gather}
From lemma 3.9 and equation (90) we get
\begin{gather}
 \vertiii{\Pi_{h}u-\varphi} \le
\Big[(C_{\epsilon}+C_{\rho h_{d}})h^{1-\delta_{0}}C'_{u}C_{1}\delta+\Big(C_{\epsilon}C_{u}(h^{1/2}+h^{2/3}+h)+  
+C_{u}(C_{\rho h_{d}}\nonumber \\
+C_{\epsilon})(h^{5/3}+h^{3/2})\Big)h^{\delta_{0}}\delta
+\Big(C_{u}(C_{\rho h_{d}}+C_{\epsilon})(h+h^{3/2})+C_{\epsilon}C_{u}h^{2}+C\Big)h^{\delta_{0}}\delta\Big].
\end{gather}
Now choosing $h$ small enough so that
\begin{gather}
\Big[(C_{\epsilon}+C_{\rho h_{d}})h^{1-\delta_{0}}C'_{u}C_{1}+\Big(C_{\epsilon}C_{u}(h^{1/2}+h^{2/3}+h)+  
+C_{u}(C_{\rho h_{d}}+C_{\epsilon})(h^{5/3}\nonumber \\
+h^{3/2})\Big)h^{\delta_{0}}
+\Big(C_{u}(C_{\rho h_{d}}+C_{\epsilon})(h+h^{3/2})+C_{\epsilon}C_{u}h^{2}+C\Big)h^{\delta_{0}}\Big]
< 1
\end{gather}
and so $\mathcal{S}$ maps $\mathcal{Q}_{\delta}(\Pi_{h}u)$ into itself.
\end{proof}
\begin{theorem}
Let $\delta > 0$ and assume that $u_{\varphi_{1}},u_{\varphi_{2}} \in \mathcal{Q}_{\delta}(\Pi_{h}u)$, then there exists a positive
constant $C$ such that the following condition holds for given $ 0 \le \delta_{0} \le 1/4 $
\begin{align}
\vertiii{\mathcal{S}u_{\varphi_{1}}-\mathcal{S}u_{\varphi_{2}}} \le C_{u}Ch^{\delta_{0}}\vertiii{u_{\varphi_{1}}-u_{\varphi_{2}}}. 
\end{align}
\end{theorem}
\begin{proof}
Consider $\delta=h^{-\delta_{0}}\vertiii{\eta}$ for some $0 \le \delta_{0} \le 1/4$, where $\eta = \Pi_{h}u-u$.
Take $\varphi_{1}=\mathcal{S}u_{\varphi_{1}}$ and $\varphi_{2}=\mathcal{S}u_{\varphi_{2}}$. Then, we have 
\begin{align}
 \bar{\mathscr{B}}(u;\varphi_{1}-\varphi_{2},v_{h})=\mathscr{F}(u_{\varphi_{1}};u_{\varphi_{1}}-u,v_{h})
 -\mathscr{F}(u_{\varphi_{2}};u_{\varphi_{2}}-u,v_{h}).
\end{align}
For proving condition (93), we first evaluate an upper bound of equation (94) as
\begin{gather}
 \Big|\mathscr{F}(u_{\varphi_{1}};u_{\varphi_{1}}-u,v_{h})
 -\mathscr{F}(u_{\varphi_{2}};u_{\varphi_{2}}-u,v_{h}) \Big| \le \nonumber \\
\Big|\sum_{K \in \mathscr{R}_{h}}\sum_{j=1}^{4}\int_{A_{j+1}CA_{j}}
(\tilde{\epsilon}_{u}(u_{\varphi_{1}})\zeta_{1}\nabla \zeta_{1}-\tilde{\epsilon}_{u}(u_{\varphi_{2}})\zeta_{2}\nabla \zeta_{2})
.\bold{n}\gamma v_{h}ds \Big|\nonumber \\
+\Big|\sum_{e \in \Gamma}\int_{e}[\gamma v_{h}]\Big\{\epsilon_{u}(u_{\varphi_{1}})\zeta_{1}\nabla \zeta_{1}
-\epsilon_{u}(u_{\varphi_{2}})\zeta_{2}\nabla \zeta_{2} \Big\}ds\Big|
+\Big|\sum_{K \in \mathscr{R}_{h}}\sum_{j=1}^{4}\int_{A_{j+1}CA_{j}} \nonumber \\
\tilde{\epsilon}_{uu}(u_{\varphi_{1}})\zeta_{1}^{2}-\tilde{\epsilon}_{uu}(u_{\varphi_{2}})\zeta_{2}^{2}\nabla u.\bold{n}\gamma v_{h}ds\Big|  \nonumber \\
+\Big|\sum_{K \in \mathscr{R}_{h}}\sum_{j=1}^{4}\int_{A_{j+1}CA_{j}}
\tilde{(\rho h_{d})}_{uu}(u_{\varphi_{1}})\zeta_{1}^{2}-\tilde{(\rho h_{d})}_{uu}(u_{\varphi_{2}})\zeta_{2}^{2}\vec{\beta}.\bold{n}\gamma v_{h}ds\Big|  \nonumber \\
+\Big|\sum_{e \in \Gamma}\int_{e}[\gamma v_{h}]\Big\{(\rho h_{d})_{uu}(u_{\varphi_{1}}) \zeta_{1}^{2}
-(\rho h_{d})_{uu}(u_{\varphi_{2}}) \zeta_{2}^{2} \Big\}ds\Big|. 
 \end{gather}
Now by using Taylor's formula we obtain
\begin{align}
 \epsilon_{u}(u_{\varphi_{1}})(u_{\varphi_{1}}-u)-\epsilon_{u}(u_{\varphi_{2}})(u_{\varphi_{2}}-u)
 =\epsilon(u_{\varphi_{1}})-\epsilon(u_{\varphi_{2}}) \nonumber \\
 =\tilde{R_{\epsilon_{1}}}(u_{\varphi_{1}},u_{\varphi_{2}})
 (u_{\varphi_{1}}-u_{\varphi_{2}}) 
\end{align}
and 
\begin{gather}
 \tilde{\epsilon}_{uu}(u_{\varphi_{1}})(u_{\varphi_{1}}-u)^{2}-\tilde{\epsilon}_{uu}(u_{\varphi_{2}})
 (u_{\varphi_{2}}-u)^{2}=\nonumber \\
 \tilde{R_{\epsilon_{2}}}(u_{\varphi_{1}},u_{\varphi_{2}})(u_{\varphi_{1}}-u_{\varphi_{2}})^{2}
 +\tilde{\epsilon}_{uu}(u_{\varphi_{2}})(u_{\varphi_{2}}-u)(u_{\varphi_{1}}-u_{\varphi_{2}}). 
\end{gather}
Now using (96) and (97) property and using similar argument of lemma 3.8 we can bound equation (95) as
\begin{gather*}
 \bar{\mathscr{B}}(u;\varphi_{1}-\varphi_{2},v_{h}) \le
 C_{\epsilon}\Big[\vertiii{\chi}^{2}
 +C_{u}\Big(h^{5/3}+h^{1/2}+h^{2/3}+h(1+h^{1/2})\Big) \nonumber \\
 \vertiii{\chi}\vertiii{u_{\varphi_{1}}-\Pi_{h}u}
 +C_{u}(h^{2}+h+h^{3/2})\vertiii{\chi}\vertiii{u_{\varphi_{2}}-\Pi_{h}u} \Big]\vertiii{v_{h}}+\nonumber \\
 C_{\rho h_{d}}\Big[\vertiii{\chi}^{2}+C_{u}(h^{5/3}+h^{3/2}) 
 \vertiii{\chi}\vertiii{u_{\varphi_{1}}-\Pi_{h}u} + C_{u}(h^{3/2}+h)\nonumber \\
 \vertiii{\chi}\vertiii{u_{\varphi_{2}}-\Pi_{h}u}\Big]\vertiii{v_{h}}
 \le CC_{u}h^{\delta_{0}}\vertiii{\chi}\vertiii{v_{h}}.
\end{gather*}
Now taking $v_{h}=\varphi_{1}-\varphi_{2}$ and using coercive property we have the desire result.
\end{proof}
\section{Error Estimates}\label{section:error}
In this section, we prove that under light load operating condition optimal order estimate in $H^{1}$ can be achieved in the 
defined norm $|\lVert . \rVert|$.
Let $u_{\mathscr{I}} \in \mathscr{V}_{h}$ be an interpolant of $u$, for which the following well known approximation property holds:
\begin{align}\label{eq:61}
 |u-u_{\mathscr{I}}|_{l,K} \le Ch^{2-l}|u|_{2,K} \quad  \forall K \in \mathscr{R}_{h}, \quad l=0,1,
\end{align}
where $C$ depends only on the angle $K$. The following theorem we will require to establish our justification.
\begin{theorem}
Suppose $u \in H^{2}(\Omega)\cap H_{0}^{1}(\Omega)$ and $u_{h} \in \mathscr{V}_{h}$ be the solution of (34). Then
there exists a constant $C$ without dependent of $h$ such that 
\begin{align}\label{eq:62}
 |\lVert u-u_{h} \rVert| \le Ch|u|_{2}
\end{align}
\end{theorem}
\subsection{$L^{2}$-Error Estimates}
In this section, $L^{2}$-error estimate is evaluated for the light load parameter case by exploiting the
Aubin-Nitsche ``trick''.
\begin{theorem}
 Let $u \in H^{2}(\Omega) \cap H^{1}_{0}(\Omega) $ and $u_{h} \in \mathscr{V}_{h}$ be the solution of problem
 \ref{eq:1} and \ref{eq:49} respectively. Then there exists a positive constant $C$ independent of $h$ such that
 \begin{gather}\label{eq:66}
  \lVert u-u_{h} \rVert \le Ch^{2} \lVert u \rVert_{2}
 \end{gather}
\end{theorem}
\begin{proof}
Consider $\phi \in H^{2}(\Omega)$ and for fix value of $u$ and $h_{d} \in H^{2}(\Omega)$
we write the adjoint problem of (1.1) as 
\begin{align}
 -\nabla\Big( \epsilon(u)\nabla \phi + \phi \epsilon_{u}\nabla u\Big)+\vec{\beta}\Big(\rho h_{d}+(\rho h_{d})_{u}\Big)\nabla \phi=e \quad \text{in } \Omega\\
 \phi=0  \quad \text{on } \partial \Omega.
 \end{align}
also we have
\begin{gather}
 \lVert e \rVert^{2}=\mathscr{B}(u;e,\phi)+
 \sum_{K \in \mathscr{R}_{h}}\sum_{j=1}^{4}\int_{A_{j+1}CA_{j}}\epsilon_{u}e\nabla.\bold{n}\gamma \phi ds
 +\sum_{e \in \Gamma}\int_{e}[\gamma \phi]\Big\{ \epsilon_{u}e \nabla u\Big\}ds \nonumber \\
 -\sum_{K \in \mathscr{R}_{h}}\sum_{j=1}^{4}\int_{A_{j+1}CA_{j}}(\rho h_{d})_{u}e\vec{\beta}.\bold{n}\gamma \phi ds-
 \sum_{e \in \Gamma}\int_{e}[\gamma \phi]\Big\{ (\rho h_{d})_{u}e \Big\}ds
\end{gather}
First term of equation (103) is rewritten as
\[\mathscr{B}(u;e,\phi)=\mathscr{B}(u;u,\phi)-\mathscr{B}(u_{h};u_{h},\phi)+\mathscr{B}(u_{h};u_{h},\phi)-\mathscr{B}(u;u_{h},\phi)\]
\[=\underbrace{\mathscr{B}(u;u,\phi-\vartheta)-\mathscr{B}(u_{h};u_{h},\phi-\vartheta)}_{I}
+\underbrace{\mathscr{B}(u_{h};u_{h},\phi)-\mathscr{B}(u;u_{h},\phi)}_{II},\]
where $\vartheta=\mathcal{I}_{h}^{k}\phi$ such that $\vartheta|_{\partial \Omega}=0$
(Here $\mathcal{I}_{h}^{k}u \in \mathscr{V}_{h} \cap H^{2}(\Omega) \cap C^{0}(\Omega)$).
We notice that \[I=\mathscr{B}(u;u,\phi-\vartheta)-\mathscr{B}(u_{h};u,\phi-\vartheta)
+\mathscr{B}(u_{h};u,\phi-\vartheta)-\mathscr{B}(u_{h};u_{h},\phi-\vartheta)\]
\begin{gather}
=\sum_{K\in \mathscr{R}_{h}}\sum_{j=1}^{4}\int_{A_{j+1}CA_{j}}(\epsilon(u)
-\epsilon(u_{h}))\nabla u.\bold{n}\gamma(\phi-\vartheta)ds
+\sum_{e \in \Gamma}\int_{e}[\gamma (\phi-\vartheta)] \nonumber \\
\Big\{(\epsilon(u)-\epsilon(u_{h}))\nabla u.\bold{n} \Big\}ds
-\sum_{K\in \mathscr{R}_{h}}\sum_{j=1}^{4}\int_{A_{j+1}CA_{j}}(\rho(u)h_{d}(x)
-\rho(u_{h})h_{d}(x))\vec{\beta}.\bold{n}\gamma(\phi-\vartheta)ds \nonumber \\
-\sum_{e \in \Gamma}\int_{e}[\gamma (\phi-\vartheta)]\Big\{(\rho(u)h_{d}(x)-\rho(u_{h})h_{d}(x))\vec{\beta}.\bold{n} \Big\}ds+ \nonumber \\
\sum_{K\in \mathscr{R}_{h}}\sum_{j=1}^{4}\int_{A_{j+1}CA_{j}}
\epsilon(u_{h})\nabla (u-u_{h}).\bold{n}\gamma(\phi-\vartheta)ds
+\sum_{e \in \Gamma}\int_{e}[\gamma (\phi-\vartheta)]\Big\{\epsilon(u_{h})\nabla (u-u_{h})\Big\}ds \nonumber \\
=J_{I_{1}}+J_{I_{2}}+J_{I_{3}}+J_{I_{4}}+J_{I_{5}}+J_{I_{6}}.
\end{gather}
First term, $J_{I_{1}}$ of equation (104) is approximated as 
\begin{gather}
|J_{I_{1}}| \le 
{\Big| \sum_{K}\langle \epsilon(u)-\epsilon(u_{h}) \nabla u, \nabla (\phi-\vartheta) \rangle \Big|}
+\Big| \sum_{K}\int_{\partial K}[\gamma(\phi-\vartheta)-(\phi-\vartheta)]\nonumber \\
(\epsilon(u)-\epsilon(u_{h}))\nabla u.\bold{n} ds \Big| 
+{\Big|\sum_{K}\langle \nabla (\epsilon(u)-\epsilon(u_{h}))\nabla u,
(\phi-\vartheta)-\gamma(\phi-\vartheta) \rangle\Big|}\nonumber \\
={J_{01}}+{J_{02}}+{J_{03}}.
\end{gather}
We bound first term, $J_{01}$ of equation (105) as
\begin{align}
\sum_{K}\Big|\int_{K}\epsilon(u)-\epsilon(u_{h}) \nabla u.\nabla (\phi-\vartheta)dx\Big|
\le C_{u}C_{\epsilon}\vertiii{e} \lVert \phi-\vartheta \rVert.\end{align}
Second term, $J_{02}$ of equation (105) is approximated bounded above as
\begin{gather}
J_{02} \le C_{u}C_{\epsilon} \sum_{K}\Big(h^{-1}\lVert \gamma(\phi-\vartheta)-(\phi-\vartheta) \rVert^{2}_{K}
+h\lVert \gamma(\phi-\vartheta)-(\phi-\vartheta) \rVert^{2}_{1,K} \Big)^{1/2}\times \lVert e \rVert \nonumber \\
\le C_{u}C_{\epsilon}\lVert \phi-\vartheta \rVert_{H^{1}(\Omega)}\vertiii{e}.
\end{gather}
Similarly, third term $J_{03}$ of equation (105) is estimated as 
\begin{align}
J_{03} \le C_{\epsilon}C_{u} \vertiii{e} \lVert \phi-\vartheta \rVert+C_{\epsilon}C_{u}\lVert \phi-\vartheta \rVert_{H^{1}(\Omega)}\vertiii{e}.
\end{align}
Using Holder's inequality and trace inequality we estimate second term, $J_{I_{2}}$ of equation (104) as 
\begin{gather}
J_{I_{2}} \le C_{\epsilon}\sum_{e \in \Gamma}\Big(\int_{e} [\gamma (\phi-\vartheta)]^{2}ds\Big)^{1/2}
\Big(\int_{e} |e|^{4}ds\Big)^{1/4} \Big(\int_{e} |\nabla u|^{4}ds\Big)^{1/4} \nonumber \\
\le C_{\epsilon}\sum_{e \in \Gamma}\Big(\int_{e}h^{-1} [\gamma (\phi-\vartheta)]^{2}ds\Big)^{1/2}
\Big(\lVert e \rVert^{4}_{L^{4}(K)}+h\lVert e \rVert^{3}_{L^{6}(K)}\lVert \nabla e \rVert_{L^{2}(K)}\Big)^{1/4} \nonumber \\
\times \Big(\lVert \nabla u \rVert^{4}_{L^{4}(K)}+h\lVert \nabla u \rVert^{3}_{L^{6}(K)}
\lVert \nabla .\nabla u \rVert_{L^{2}(K)}\Big)^{1/4} \nonumber \\
\le C_{u}C_{\epsilon}\vertiii{e}^{2}\vertiii{\phi-\vartheta}.
\end{gather}
By using Similar argument we bound the following terms as
\begin{align}
|J_{I_{3}}| \le C_{u} \vertiii{e} \vertiii{\phi-\vartheta}, \\
|J_{I_{4}}| \le C_{u} \vertiii{e} \vertiii{\phi-\vartheta}, \\
|J_{I_{5}}| \le C_{u} \vertiii{e} \vertiii{\phi-\vartheta}, \\
|J_{I_{6}}| \le C_{u} \vertiii{e} \vertiii{\phi-\vartheta}.
\end{align}
We note that 
\begin{gather}
II=\sum_{K \in \mathscr{R}_{h}}\sum_{j=1}^{4}\int_{A_{j+1}CA_{j}}(\epsilon(u_{h})-\epsilon(u))\nabla u_{h}.\bold{n}\gamma \phi ds
+\sum_{e \in \Gamma}\int_{e}[\gamma \phi]\Big\{(\epsilon(u_{h}) \nonumber \\
-\epsilon(u))\nabla u_{h}\Big\}ds
-\sum_{K\in \mathscr{R}_{h}}\sum_{j=1}^{4}\int_{A_{j+1}CA_{j}}
(\rho(u_{h})h_{d}(x)-\rho(u)h_{d}(x))\vec{\beta}.\bold{n}\gamma \phi ds \nonumber \\
-\sum_{e \in \Gamma}\int_{e}[\gamma \phi]\Big\{\rho(u_{h})h_{d}(x)-\rho(u)h_{d}(x))\vec{\beta}.\bold{n} \Big\}ds 
=\sum_{K \in \mathscr{R}_{h}}\sum_{j=1}^{4}\int_{A_{j+1}CA_{j}}(\epsilon(u_{h}) \nonumber \\
-\epsilon(u))\nabla (u_{h}-u).\bold{n}\gamma \phi ds
+\sum_{e \in \Gamma}\int_{e}[\gamma \phi]\Big\{(\epsilon(u_{h})-\epsilon(u))\nabla (u_{h}-u)\Big\}ds \nonumber \\
+\sum_{K \in \mathscr{R}_{h}}\sum_{j=1}^{4}\int_{A_{j+1}CA_{j}}(\epsilon(u_{h})-\epsilon(u))\nabla u.\bold{n}\gamma \phi ds
+\sum_{e \in \Gamma}\int_{e}[\gamma \phi]\Big\{(\epsilon(u_{h})-\epsilon(u))\nabla u\Big\}ds \nonumber \\
-\sum_{K\in \mathscr{R}_{h}}\sum_{j=1}^{4}\int_{A_{j+1}CA_{j}}
(\rho(u_{h})h_{d}(x)-\rho(u)h_{d}(x))\vec{\beta}.\bold{n}\gamma \phi ds \nonumber \\
-\sum_{e \in \Gamma}\int_{e}[\gamma \phi]\Big\{\rho(u_{h})h_{d}(x)-\rho(u)h_{d}(x))\vec{\beta}.\bold{n} \Big\}ds \nonumber \\
=J_{II_{1}}+J_{II_{2}}+J_{II_{3}}+J_{II_{4}}+J_{II_{5}} +J_{II_{6}} 
\end{gather}
First term $J_{II_{1}}$ of equation (114) is approximated as 
\begin{gather}
J_{II_{1}} \le \Big|\sum_{K} \langle \epsilon(u_{h})-\epsilon(u) \nabla (u_{h}-u), \nabla \phi \rangle \Big|
+\Big|\sum_{e \in \Gamma}\int_{\partial K}[\gamma \phi- \phi]\Big\{ (\epsilon(u_{h}) \nonumber \\
-\epsilon (u))\nabla (u_{h}-u) \Big\} ds\Big|
+\Big|\sum_{K} \langle \nabla (\epsilon(u_{h})-\epsilon(u)) \nabla (u_{h}-u), \phi -\gamma \phi \rangle \Big| \nonumber \\
=J^{1}_{II_{1}}+J^{2}_{II_{1}}+J^{3}_{II_{1}}.
\end{gather}
First term $J^{1}_{II_{1}}$ of equation (115) is estimated by using holder's inequality
\begin{gather}
 J^{1}_{II_{1}} \le C_{u} \Big(\sum_{K}\int_{K} |e|^{3} dx \Big)^{1/3} \Big(\sum_{K}\int_{K} |\nabla e|^{2} dx \Big)^{1/2}
\Big(\sum_{K}\int_{K} |\nabla \phi|^{6} dx \Big)^{1/6} \nonumber \\
 \le C_{u}C\vertiii{e}^{2}\lVert \phi \rVert_{H^{2}(\Omega)}
\end{gather}
Using trace inequality second term $J^{2}_{II_{1}}$ of equation (115) is estimated as 
\begin{gather}
J^{2}_{II_{1}} \le C_{u}\Big(\int_{\partial K}[\gamma \phi-\phi]^{2}ds\Big)^{1/2}
\Big(\int_{\partial K}|e|^{4}ds\Big)^{1/4}
\Big(\int_{\partial K}|\nabla e|^{4}ds\Big)^{1/4} \nonumber \\
 \le C_{u}C\vertiii{e}^{2}\lVert \phi \rVert_{H^{2}(\Omega)}
\end{gather}
Third term, $J^{3}_{II_{1}}$ of equation (115) is bounded using Holder's and trace inequality as 
\begin{gather}
J^{3}_{II_{1}} \le C_{u} \Big(\sum_{K}\int_{K} |e|^{3} dx \Big)^{1/3} \Big(\sum_{K}\int_{K} |\nabla e|^{2} dx \Big)^{1/2}
\Big(\sum_{K}\int_{K} |\nabla (\phi-\gamma \phi)|^{6} dx \Big)^{1/6}\nonumber \\
+C_{u}\Big(\int_{\partial K}|\gamma \phi-\phi|^{2}ds\Big)^{1/2}
\Big(\int_{\partial K}|e|^{4}ds\Big)^{1/4}
\Big(\int_{\partial K}|\nabla e|^{4}ds\Big)^{1/4}\nonumber \\
\le  C_{u}C\vertiii{e}^{2}\lVert \phi \rVert_{H^{2}(\Omega)}
\end{gather}
We bound the second term $J_{II_{2}}$ of equation (114) by using trace as well as Holder's inequality to obtain
\begin{gather}
J_{II_{2}} \le C_{u}C\vertiii{e}^{2}\lVert \phi \rVert_{H^{2}(\Omega)}.
\end{gather}
Now consider the third term of equation (114) and take second term of equation (103) and using Taylor's formula get 
\begin{gather}
\Big|\sum_{K \in \mathscr{R}_{h}}\sum_{j=1}^{4}\int_{A_{j+1}CA_{j}}
\tilde{\epsilon}_{uu}(u_{h})e^{2}\nabla u.\bold{n}\gamma \phi ds\Big| \le C_{u}C_{\epsilon}\vertiii{e}^{2}
\Vert \phi\Vert_{H^{2}(\Omega)}.
\end{gather}
Take fourth term of equation (114) and third term of equation (103) and use Taylor's formula to obtain
\begin{gather}
\Big|\sum_{e \in \Gamma}\int_{e}
[\gamma \phi]\Big\{\tilde{\epsilon}_{uu}(u_{h})e^{2}\nabla u \Big\} ds\Big| \le C_{u}C_{\epsilon}\vertiii{e}^{2}
\Vert \phi\Vert_{H^{2}(\Omega)}.
\end{gather}
We take fifth term of equation (114) and fourth term of equation (103) and use Taylor's formula to get
\begin{gather}
\Big|\sum_{K\in \mathscr{R}_{h}}\sum_{j=1}^{4}\int_{A_{j+1}CA_{j}}
\tilde{\rho h_{d}}_{uu}e^{2}\vec{\beta}.\bold{n}\gamma \phi ds \Big| \le
C_{u}C_{\rho h_{d}}\vertiii{e}^{2}
\Vert \phi\Vert_{H^{2}(\Omega)}
\end{gather}
Finally taking sixth term of equation (114) and fifth term of equation (103) and by using taylor's formula we get bound as
\begin{gather}
\Big|\sum_{e \in \Gamma}\int_{e}
[\gamma \phi]\Big\{\tilde{\rho h_{d}}_{uu}(u_{h})e^{2}\Big\} ds\Big| \le C_{\rho h_{d}}\vertiii{e}^{2}
\Vert \phi\Vert_{H^{2}(\Omega)}.
\end{gather}
\end{proof}
\section{Numerical test of Discontinuous Galerkin finite volume method}\label{section:ntest}
In this section, numerical experiments are performed for EHL point contact cases. Optimal error estimates 
for pressure $(u-u_{h})$ are achieved in broken $H^{1}$ norm $|\lVert . \rVert|$ and $L^{2}$ norm which
are plotted in Fig.~\ref{fig:err} with the red line and the blue line respectively.
Numerical results confirm the theoretical order of convergence derived in Theorem
4.1 and Theorem 4.3 which are almost equal to 1 and 2 respectively.
We have also shown graphical figures of pressure $u$ Fig.~\ref{fig:lp} and Fig.~\ref{fig:mp}
and film thickness $H$ Fig.~\ref{fig:h1} under light load condition by writing in Moe’s non-dimensional parameter
form detail can be found in \cite{moes}.
\begin{figure}[ht!]
\centering
\includegraphics[width=3.5in, height=3.5in, angle=0]{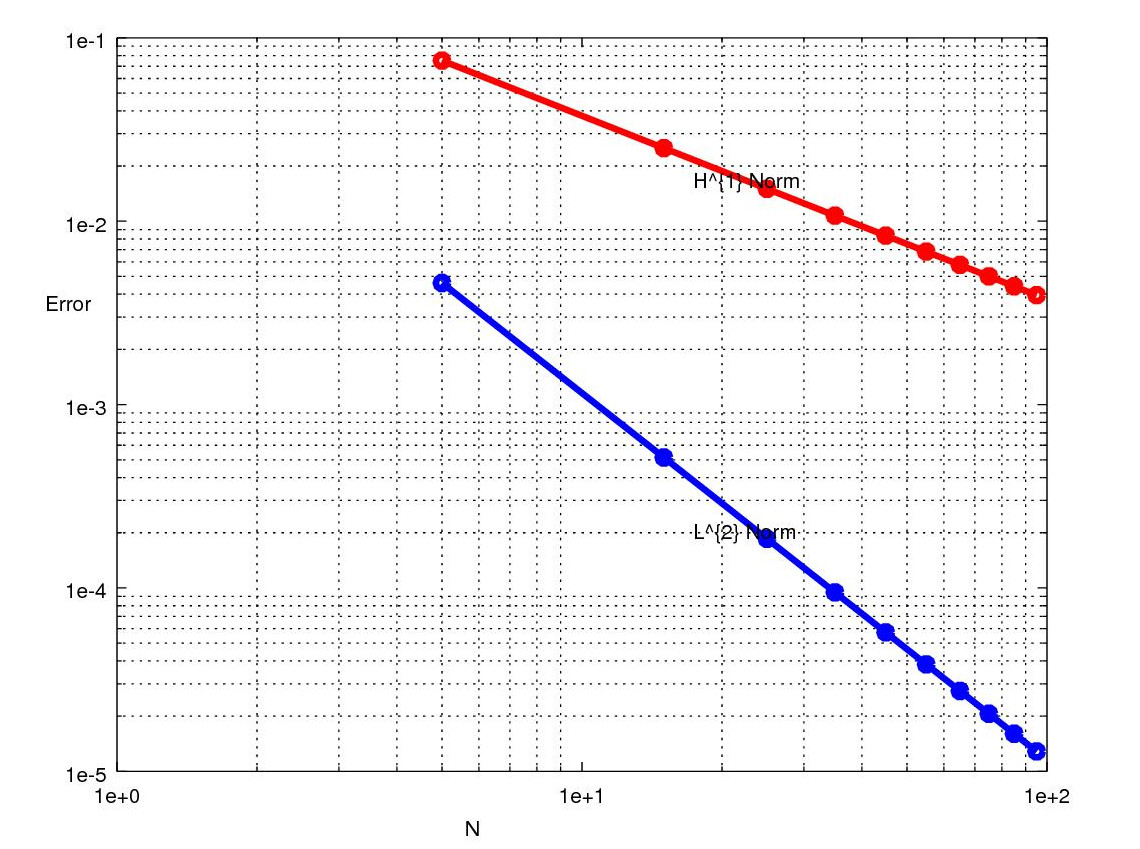}
\caption{$L^{2}$ (in blue line) and $H^{1}$ (in red line) error $\lVert u-u_{h}\rVert$ plot}
\label{fig:err}
\end{figure}
\begin{figure}[ht!]
\centering
\includegraphics[width=3.5in, height=1.5in, angle=0]{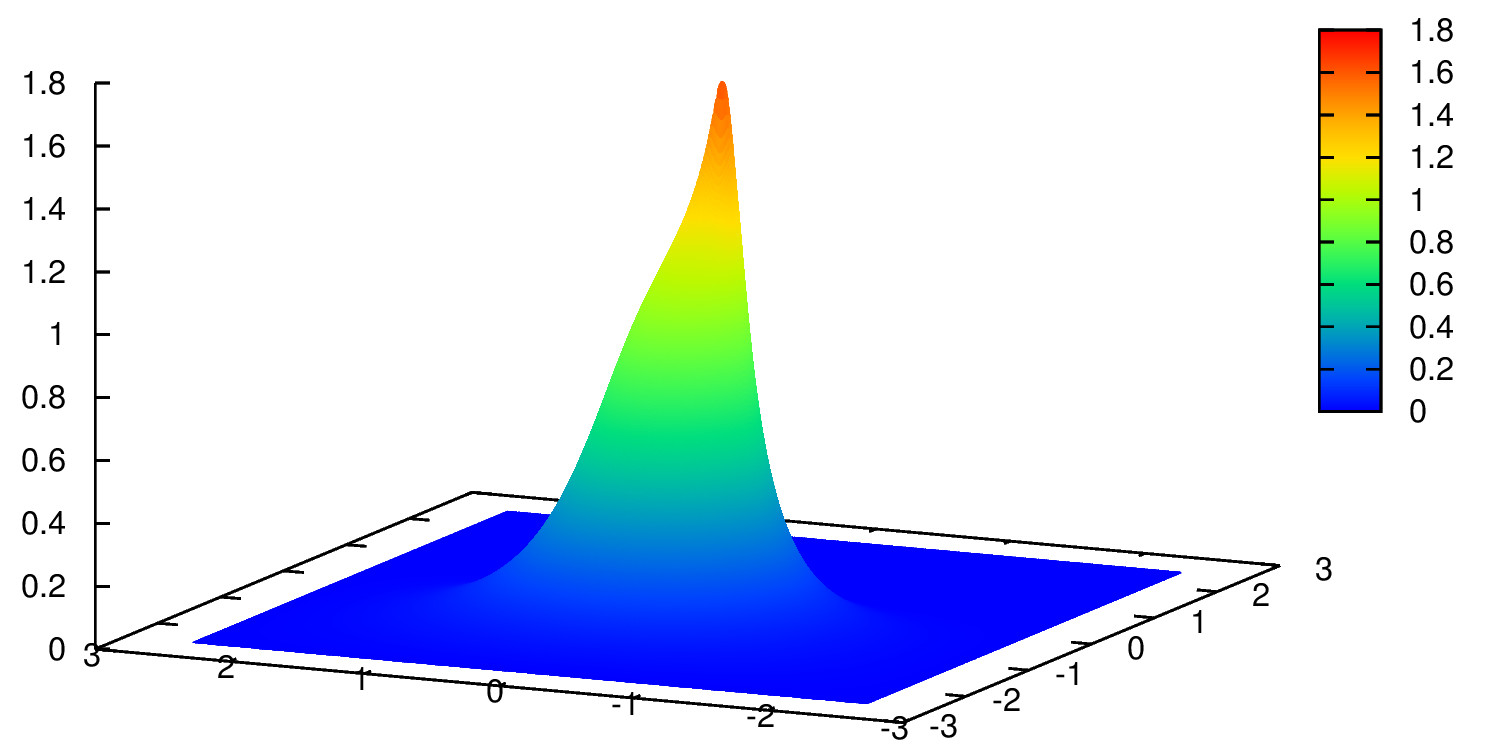}
\caption{Pressure profile for light load case $M=7$ and $L=10$ (Moe's parameter)}
\label{fig:lp}
\end{figure}
\begin{figure}[ht!]
\centering
\includegraphics[width=3.5in, height=1.5in, angle=0]{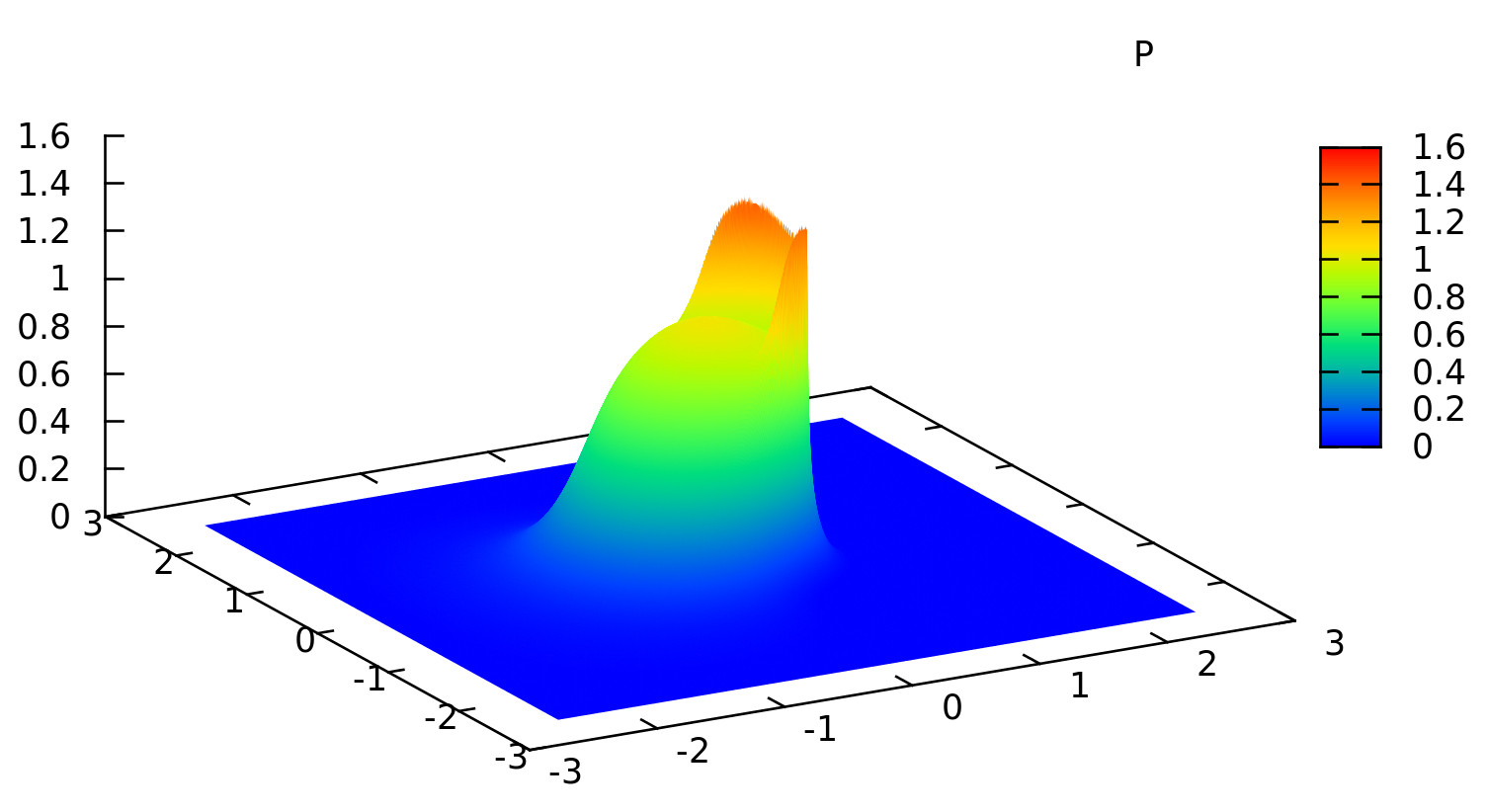}
\caption{Pressure profile for moderately high load case $M=20$ and $L=10$}
\label{fig:mp}
\end{figure}
\begin{figure}[ht!]
\centering
\includegraphics[width=3.5in, height=1.5in, angle=0]{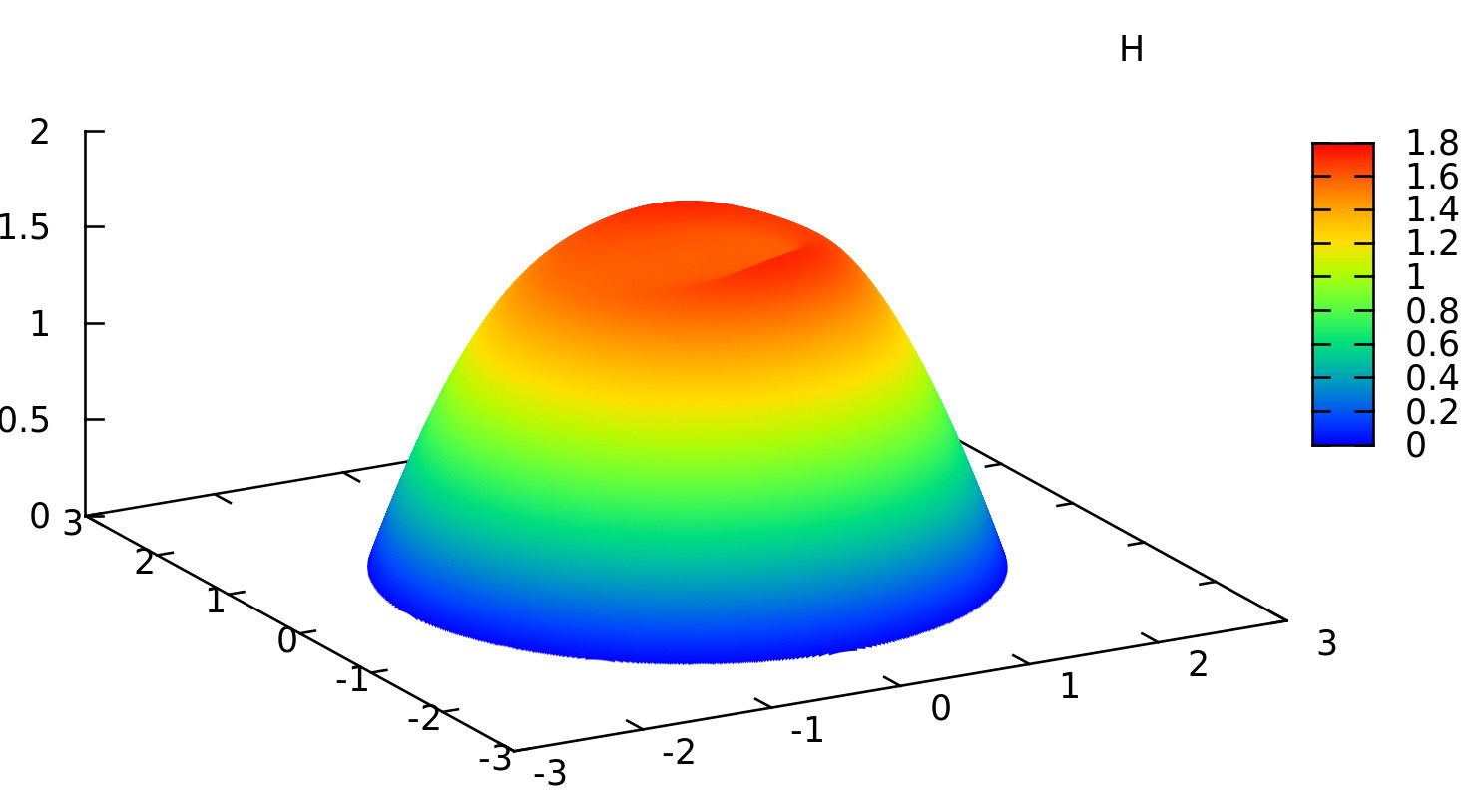}
\caption{Film thickness profile in inverted form for $M=20$ and $L=10$}
\label{fig:h1}
\end{figure}
\subsection{Film thickness calculation}
Accurate film thickness $H$ computation is very important for stable relaxation procedure and require extra care in its computation.
Film thickness calculation is calculated as follows
\small
\begin{align}\label{eq:74}
 {h_{d}(x,y)} ={h_{0}+\frac{{x}^2+{y}^2}{2}+ \frac{2}{\pi^2}\int_{x_{-}}^{x_{+}}\int_{y_{-}}^{y_{+}}\frac{p(x',y')dx'dy'}{\sqrt{(x-x')^2+(y-y')^2}}}
 \end{align}
 \normalsize
\begin{align}\label{eq:75}
 =h_{0}+\frac{{x}^2+{y}^2}{2}+ \frac{2}{\pi^2}\sum^{N}_{e=1} \int_{e}\frac{\sum^{p_e+1}_{i=0}a^{e}_{i}\mathscr{N}^{e}_{i}(x',y')}{\sqrt{(x-y')^2+(y-y')^2} } dx'dy'
\end{align}
\begin{align}\label{eq:76}
 =h_{0}+\frac{{x}^2+y^2}{2}+\frac{2}{\pi^2}\sum^{N}_{e=1} \sum^{p_e+1}_{i=0} \int_{e}\frac{a^{e}_{i} \mathscr{N}^{e}_{i}(x',y')  dx'dy' }{\sqrt{(x-y')^2+(y-y')^2}} 
\end{align}
\begin{align}\label{eq:77}
=h_{0}+\frac{{x}^2+y^2}{2}+ \frac{2}{\pi^2}\sum^{N}_{e=1} \sum^{p_e+1}_{i=0}\mathscr{G}^{e}_{i}(x,y)a^{e}_{i}
\end{align}
\subsection{Mild singular integral computation}
Singularity at $(x',y') = (x, y)$ can be approximated in the following manner. We first rewrite kernel  $\mathscr{G}^e_i(x)$ in the following form
\begin{align}\label{eq:78}
\mathscr{G}^{e}_{i}(x) = \int_{\Omega_{e}}\frac{\mathscr{N}^{e}_{i}(x',y')dx'dy'}{\sqrt{(x-y')^2+(y-y')^2}} \nonumber \\
 = \frac{h_{x}^{e}}{2}\frac{h_{y}^{e}}{2} \int_{-1}^{1}\int_{-1}^{1}\frac{\mathscr{N}^{e}_{i}(x'(\xi,\chi ),y'(\xi,\chi ))
 d\xi d\chi }{\sqrt{(x-x'(\xi,\chi ))^2+(y-y'(\xi,\chi ))^2}}\nonumber\\
 \approx\frac{h_{x}^{e}}{2}\frac{h_{y}^{e}}{2}\sum_{j=1}^{m}\sum_{k=1}^{m}
 \frac{\mathscr{N}^{e}_{i}(x'(\xi_{j},\chi _{k}),y'(\xi_{j},\chi _{k})) w_{j} w_{k}}{\sqrt{(x-x'(\xi_{j},\chi _{k}))^2+(y-y'(\xi_{j},\chi _{k}))^2}},
\end{align}
where $h_{x}^{e}=x_{2}-x_{1}$ and $h_{y}^{e}=y_{2}-y_{1}$ are the step sizes of element $e$ in the $x$ direction and $y$ direction respectively 
and $\xi \in [-1,1]$ and $\chi  \in [-1,1]$ are the coordinate directions for the reference element. We have applied here $m$ point quadrature in $x$ and $y$ direction of discretization.
Singular quadrature procedure is implemented here to resolve the singularity appeared in term $\mathscr{G}_{i}^{e}(x,y)=\frac{1}{\sqrt{(x-x')^{2}+(y-y')^{2}}}$ at the point $(x,y)$.
Idea involve by dividing the element $e$ into four subpart elements $\mathscr{F}_{k}, k=1,2,3,4$ for calculating integrals of $\mathscr{G}_{i}^{{\mathscr{F}_{k}}}(x,y)=\frac{1}{\sqrt{(x-x')^{2}+(y-y')^{2}}}$.
Each four integrals have chosen in a such way that they have only one singular point in the domain of integration.
Four integrals defined above can be evaluated as in general integral form:
\begin{align}\label{eq:79}
 \mathscr{S}^{*} =\int_{0}^{1}\int_{0}^{1} \mathscr{F}^{*}(x,y)\mathscr{G}^{*}(x,y)dxdy,
\end{align}
where $\mathscr{F}^{*}$ is analytic function and $\mathscr{G}^{*}$ is a function having a mild singularity at only one point.
\begin{align}\label{eq:80}
 \mathscr{S}\approx \mathscr{S}_{n}^{*} = \sum_{i=1}^{n} \mathscr{I}_i,
\end{align}
where
\begin{align}\label{eq:81}
\mathscr{I}_i = \int_{x_i}^{x_{i-1}} \int_{y_i}^{y_{i-1}}\mathscr{F}^{*}(x,y)\mathscr{G}^{*}(x,y)dxdy , (i \geq 1).
\end{align}
Where $ (x_0,y_{0}) = (1,1) $ and  $(x_{n},y_{n}) \rightarrow (0,0) \quad \text{as} \quad  n \rightarrow \infty$ for the value 
$(x_n,y_{n}) = (\theta^{n},\theta^{n}) , (0 < \theta < 1)$.
\subsection{Load balance equation calculation}
The force balance equation is discretized according to:
\begin{align}\label{eq:82}
\sum^{N}_{e=1}\int_{\Omega_e} \sum^{p_e+1}_{i=0}{\mathscr{G}_{1}}^{e}_{i} (x,y)a^{e}_{i} dx dy-\frac{2\pi}{3}=0
\end{align}
By introducing another kernel $\mathscr{N_{1}}^{e}_{i}$ \
\begin{align}\label{eq:83}
{\mathscr{G}_{1}}^{e}_{i} = \int_{\Omega_e} \mathscr{N_{1}}^{e}_{i}(x,y)dx dy
\end{align}
the discrete force balance equation can be rewritten as:
\begin{align}\label{eq:84}
\sum^{N}_{e=1} \sum^{p_e+1}_{i=0} (\mathscr{N_{1}}^{e}_{i})a^{e}_{i}-\frac{2\pi}{3}=0
\end{align}
\section{Conclusion}\label{section:con}
New discontinuous Galerkin finite volume method is developed and analyzed
with the help of interior-exterior penalty approach. The method is fully systematic
and easily parallelized in MPI (Massage passing interface) environment.
Stability estimates are proved by showing operator as pseudo-monotone for moderate
load condition. Optimal error estimates are achieved under light load condition
theoretically as well as by numerical computation in $H^{1}$ and $L^2$ norm respectively.
More implementation issues and applications will be discussed in the second part of the paper.
\begin{appendices}
\section{Relaxation of EHL}\label{appendix:relax}
For finding unique solution we can update our nonlinear operator iterative manner by taking old and new pressure value in the following form
\begin{align}\label{eq:87}
U_{\text{new}} = U_{\text{old}}+\Big(\frac{\partial \mathscr{T}_{d}(U)}{\partial U}\Big)^{-1}\mathcal{R}_{s},                                                           
\end{align}
where $\mathcal{R}_{s}$ is the numerical residual value of the discretized Reynolds equation and, $\mathscr{T}_{d}$
is discretized nonlinear operator. The approximation of $\frac{\partial \mathscr{T}_{d}(U)}{\partial U}$ can be evaluated in the following way,
\[ \frac{\partial \mathscr{T}_{d}(U)}{\partial U}\approx \frac{\partial \mathscr{T}^{*}_{d}(U)}{\partial U}-\frac{\partial \mathscr{T}^{**}_{d}(U)}{\partial U} \]
\begin{align}\label{eq:88}
\approx \mathscr{A}^{*}_{d}(U)-\frac{\partial \mathscr{T}^{**}_{d}(U)}{\partial U}
\end{align}
In the above equation \ref{eq:88}, we can notice that term $\frac{\partial \mathscr{T}^{**}_{d}(U)}{\partial U}$ is a full dense matrix and evaluated in the following way,
\small
\[ \frac{\partial \mathscr{T}^{**}_{d}(U)}{\partial U}\Big|_{I,J} =\sum\limits_{K \in \mathscr{R}_h} \sum\limits_{j=1}^{3} \int_{A_{j+1}CA_{j}}(\rho \frac{\partial h_{d}}{\partial U_{j}^{f}}+h_{d}\frac{\partial \rho}{\partial U_{j}^{f}}).(\beta .\bold{n})\gamma v ds\]
\[ + \sum\limits_{e \in \Gamma} \int_{e} [\gamma v] \Big\{ (\rho \frac{\partial h_{d}}{\partial U_{j}^{f}}+h_{d}\frac{\partial \rho}{\partial U_{j}^{f}}).(\beta.\bold{n})\Big\} ds, \]
where the $I^{\text{th}}$ subscript denote the row generated with the test function $v=\mathscr{N}^{e}_{i}(\bold{X})$ and the $J^{\text{th}}$ subscript correspond to the unknown $U_{j}^{{\mathscr{F}}}$.
According to the equation (60) we can evaluate the following expression
\[ \frac{\partial h_{d}}{\partial U_{j}^{f}}=\mathscr{G}_{j}^{f}     \]
\normalsize
which can be pre-evaluated. It is worth mentioning that, from equation (60) the film thickness depends heavily on the local pressure and very less on the pressure for away. The value of $\mathscr{G}_{j}^{{\mathscr{F}}}$
is rapidly decreases as the position of element $\mathscr{F}$ is far away from the position of $\bold{X}=(x,y)$. From the above information we can reduce our computation cost by considering the following approximations of 
$\frac{\partial \mathscr{T}^{**}_{d}(U)}{\partial U} $:
\begin{itemize}
 \item $\frac{\partial h_{d}(\bold{X})}{\partial U_{j}^{f}}=0$ where $\bold{X}\in e$ if $f\neq e$ and $f$ is not a adjacent element of $e$.
 \item $\frac{\partial h_{d}(\bold{X})}{\partial U_{j}^{f}}=0$ where $\bold{X}\in \Gamma_{\text{int}}$ and if  $f$ is not a adjacent element of $\Gamma_{\text{int}}$.
 \item $\frac{\partial h_{d}(\bold{X})}{\partial U_{j}^{f}}=0$ where $\bold{X}\in \Gamma_{D}$ and if  $f$ is not a adjacent element of $\Gamma_{D}$.
 \item $\frac{\partial h_{d}(\bold{X})}{\partial U_{j}^{f}}=\mathscr{G}_{j}^{f}(\bold{X})$, otherwise.
\end{itemize}
\section{Parameters used in computation}\label{appendix:pvalue}
Following Parameters relation is defined in our study for $\epsilon^{*},\rho(u), \eta(u),\lambda.$
\begin{align}\label{eq:89}
\rho(u) = \rho_{0}\frac{l_{1} +1.34 u}{l_{1}+u}\nonumber\\
\eta(u) = \eta_{0}e^{l_{2}u}\nonumber\\
 \epsilon^{*} = \frac{\rho h^{3}_{d}}{\eta \lambda}\nonumber\\
 \lambda = \frac{12\mu v(2R)^{3}}{\pi E}
\end{align}
where $l_{1}=0.59\times 10^{9}$ and $l_{2}\approx 2.0\times 10^{-8}$.
\end{appendices}
\section*{Acknowledgment}
This work is fully funded by DST-SERB
Project reference no.PDF/2017/000202 under N-PDF fellowship program 
and working group at the Tata Institute of Fundamental Research, TIFR-CAM, Bangalore. 
The authors also would like to thank to Department of Mathematics \& Statistics Indian Institute of Technology Kanpur for 
their lodging support during writing this article.
\bibliographystyle{plain}
\bibliography{ref}
\end{document}